\documentclass[a4paper,11pt]{amsart}

\usepackage{mathtools}
\usepackage{color}
\usepackage{comment}
\usepackage{amssymb}
\usepackage{amsfonts}
\usepackage{amscd}

\usepackage{slashed}
\usepackage{pdflscape}
\usepackage{tikz}
\usepackage{enumerate}
\usepackage{multirow}
\usepackage[backref,pagebackref,linkcolor=blue]{hyperref}
\usepackage{mathrsfs}
\usepackage{pifont}
\DeclareRobustCommand{\SkipTocEntry}[5]{}

\newcommand{\bm}[1]{\mbox{\boldmath~$ #1~$}}

\newtheorem{theorem}{Theorem}[section]
\newtheorem{lemma}[theorem]{Lemma}
  
\newtheorem{proposition}[theorem]{Proposition}
\newtheorem{corollary}[theorem]{Corollary}

\theoremstyle{definition}
\newtheorem{definition}[theorem]{Definition}

\theoremstyle{remark}
\newtheorem{remark}[theorem]{Remark}

\newtheorem{problem}[theorem]{Problem}

\usepackage{pdfsync}

\usepackage{graphicx} 
\usepackage{amsmath} 
\usepackage{amsfonts}
\usepackage{amssymb}

\renewcommand{\O}{{\mathcal O}}

\newcommand{\Nint}{{\mathbb N}}

\newcommand{\be}{\begin{equation}}

\newcommand{\ee}{\end{equation}}

\newcommand{\sib}{\bar{\si}}

\newcommand{\cB}{{\mathcal B}}

\newcommand{\Z}{{\mathcal Z}}

\newcommand{\nablab}{\bar\nabla}

\newcommand{\cc}{\boldsymbol{c}}
\newcommand{\dd}{\boldsymbol{d}}

\newcommand{\II}{{\bf\rm I\hspace{-.2mm}I}}
\newcommand{\IIo}{\mathring{{\bf\rm I\hspace{-.2mm} I}}{\hspace{.2mm}}}

\newcommand{\aNd}{\mbox{\boldmath$ \nabla$}\hspace{-.6mm}}

\newcommand{\si}{\sigma}

\newcommand{\ba}{\begin{array}}

\newcommand{\ea}{\end{array}}

\newcommand{\beq}{\begin{eqnarray}}

\newcommand{\eeq}{\end{eqnarray}}

\newtheorem{lm}{lemma}

\newtheorem{thee}{theorem}

\newtheorem{proo}{proposition}

\newtheorem{co}{corollary}

\newtheorem{rem}{remark}

\newtheorem{deff}{definition}

\newcommand{\bd}{\begin{deff}}

\newcommand{\ed}{\end{deff}}

\newcommand{\bl}{\begin{lm}}

\newcommand{\el}{\end{lm}}

\newcommand{\bp}{\begin{proo}}

\newcommand{\ep}{\end{proo}}

\newcommand{\bt}{\begin{thee}}

\newcommand{\et}{\end{thee}}

\newcommand{\bc}{\begin{co}}

\newcommand{\ec}{\end{co}}

\newcommand{\brm}{\begin{rem}}

\newcommand{\erm}{\end{rem}}

\newcommand{\F}{\overline{F}}

\usepackage{t1enc}
\def\frak{\mathfrak}

\def\Cal{\mathcal}

\newcommand{\newc}{\newcommand}

\let\ccdot\cdot
\def\cdot{\hbox to 2.5pt{\hss$\ccdot$\hss}}

\newcommand{\Si}{\Sigma}

\newc{\aR}{\mbox{\boldmath{$ R$}}}
\newc{\aS}{\mbox{\boldmath{$ S$}}}
\newc{\aT}{\mbox{\boldmath{$ T$}}}
\newc{\aW}{\mbox{\boldmath{$ W$}}}
\newcommand{\aX}{\mbox{\boldmath{$ X$}}\hspace{-.2mm}}

\newc{\aD}{\mbox{\boldmath{$ D$}}\hspace{-.2mm}}

\newc{\aK}{\mbox{\boldmath{$ K$}}}
\newc{\aL}{\mbox{\boldmath{$ L$}}}
\newcommand{\ce}{{\Cal E}}

\newcommand{\cO}{{\Cal O}}

\newcommand{\ct}{{\Cal T}}

\newcommand{\hD}{\widehat{D}}

\usepackage{amssymb}
\usepackage{amscd}

\newcommand{\nd}{\nabla}

\newcommand{\Rho}{{\rm P}}
\newcommand{\Up}{\Upsilon}

\newcommand{\Ric}{\operatorname{Ric}}
\newcommand{\Sc}{\operatorname{Sc}}

\newcommand{\bs}[1]{\boldsymbol{#1}}

\newcommand{\cT}{{\mathcal T}}

\let\i=\iota

\newcommand{\nn}[1]{(\ref{#1})}

\newcommand{\sX}{{\mbox{\scriptsize\boldmath{$X$}}}} 
\newcommand{\bg}{\mbox{\boldmath{$ g$}}}

\newcommand{\J}{{\rm J}}
\newc{\obstrn}[2]{B^{#1}_{#2}}

\newcommand{\rpl}                         % +) or <+
{\mbox{$
\begin{picture}(12.7,8)(-.5,-1)
\put(0,0.2){$+$}
\put(4.2,2.8){\oval(8,8)[r]}
\end{picture}$}}

\newcommand{\lpl}                         % (+ or +>
{\mbox{$
\begin{picture}(12.7,8)(-.5,-1)
\put(2,0.2){$+$}
\put(6.2,2.8){\oval(8,8)[l]}
\end{picture}$}}

\usepackage{ifthen}

\newc{\tensor}[1]{#1}
\newc{\Mvariable}[1]{\mbox{#1}}
\newc{\down}[1]{{}_{#1}}
\newc{\up}[1]{{}^{#1}}

\newc{\JulyStrut}{\rule{0mm}{6mm}}
\newc{\midtenPan}{\mbox{\sf S}}
\newc{\midten}{\mbox{\sf T}}
\newc{\midtenEi}{\mbox{\sf U}}
\newc{\ATen}{\mbox{\sf E}}
\newc{\BTen}{\mbox{\sf F}}
\newc{\CTen}{\mbox{\sf G}}

\newcommand{\w}{\mbox{\bf w}} 

\def\sideremark#1{\ifvmode\leavevmode\fi\vadjust{\vbox to0pt{\vss
 \hbox to 0pt{\hskip\hsize\hskip1em
 \vbox{\hsize2cm\tiny\raggedright\pretolerance10000
  \noindent #1\hfill}\hss}\vbox to8pt{\vfil}\vss}}}

\numberwithin{equation}{section}

\renewcommand\F{{\mathcal F}}

\newcommand{\B}{\mathcal B}
\newcommand{\Zint}{{\mathbb Z}}

\newcommand{\Pop}{{\sf P}}

\begin{document}

\renewcommand{\today}{}
\title{Conformal hypersurface geometry via a  boundary Loewner--Nirenberg--Yamabe problem}
\author{A. Rod Gover${}^{\mathfrak G}$ \& Andrew Waldron${}^{\mathfrak W}$}

\address{${}^{\mathfrak G}$Department of Mathematics\\
  The University of Auckland\\
  Private Bag 92019\\
  Auckland 1142\\
  New Zealand} \email{r.gover@auckland.ac.nz}
  
  \address{${}^{\mathfrak W}$Department of Mathematics\\
  University of California\\
  Davis, CA95616, USA} \email{wally@math.ucdavis.edu}

\vspace{10pt}

\renewcommand{\arraystretch}{1}

\begin{abstract}

We develop a new approach to the conformal geometry of embedded
hypersurfaces by treating them as conformal infinities of conformally
compact manifolds. This involves the Loewner--Nirenberg-type problem
of finding on the interior a metric that is both conformally compact
and of constant scalar curvature. Our first result is an
asymptotic solution to all orders. This involves log terms. We show
that the coefficient of the first of these is a new hypersurface
conformal invariant which
% is interesting for several reasons, the first
%of which is that it 
generalises to higher dimensions the important Willmore
invariant of embedded surfaces. We call this the obstruction density.
For even dimensional hypersurfaces it is a fundamental curvature
invariant.  We make the latter notion precise and show that the
obstruction density and the trace-free second fundamental form are, in
a suitable sense, the only such invariants.  We also show that this
obstruction to smoothness is a scalar density analog of the
Fefferman--Graham obstruction tensor for Poincar\'e--Einstein metrics;
in part this is achieved by exploiting Bernstein--Gel'fand--Gel'fand
machinery.  The solution to the constant scalar curvature problem
provides a smooth hypersurface defining density determined canonically
by the embedding up to the order of the obstruction. We give two key
applications: the construction of conformal hypersurface invariants
and the construction of conformal differential operators.  In
particular we present an infinite family of conformal powers of the
Laplacian determined canonically by the conformal embedding.  In
general these depend non-trivially on the embedding and, in contrast
to Graham--Jennes--Mason--Sparling operators intrinsic to even
dimensional hypersurfaces, exist to all orders. These extrinsic
conformal Laplacian powers determine an explicit holographic formula
for the obstruction density.
\vspace{2cm}

\noindent
%\begin{keywords}

\end{abstract}

\maketitle

\pagestyle{myheadings} \markboth{Gover \& Waldron}{Conformal Hypersurface Geometry}

\newpage

\tableofcontents
\newpage

\section{Introduction}

%% \edz{Mention various results, Willmores from the old 4.4 and Energy
%%   Functionals from the old section 9}

A smoothly embedded codimension-1 submanifold~$\Sigma$ of a smooth,
$d$ dimensional mani\-fold~$M$, is termed a {\em hypersurface}. These
are critically important in geometry and analysis, not least because
manifold and domain boundaries are examples; almost any boundary
problem calls on some aspect of hypersurface geometry. Despite this, a
general theory of natural invariants, differential operators and
functionals for conformal hypersurface geometries is lacking.  We
 show here that a certain boundary problem provides a unified approach
to such problems and a number of new results.

For any continuous geometry, understanding the existence and
construction of local invariants is a critical first step.  Given a
hypersurface~$\Sigma$ in a Riemannian manifold $(M,g)$, local
conformal invariants are the natural scalar or tensor-valued fields
determined by the data~$(M,g,\Sigma)$, which have the additional
property that they are (as densities) unchanged when the metric~$g$ is
replaced by a conformally related metric, that is~$\widehat{g}$ where
$\widehat{g}=\Omega^2g$ for some positive function~$\Omega$. The
precise definitions are given in Sections~\ref{definv-sec} and
\ref{chi}.

For surfaces in Riemannian 3-manifolds an extremely interesting
conformal invariant appears as the left-hand-side of the Willmore
equation,
\begin{equation}\label{Wore}
\bar{\Delta} H +2 H(H^2-{\rm K})=0\, ,
\end{equation}
for an embedded surface~$\Sigma$ in
Euclidean~3-space~$\mathbb{E}^3$~\cite{Willmore}. Here $H$ and~${\rm
  K}$ are, respectively, the mean and Gau\ss\ curvatures,
while~$\bar\Delta$ is the Laplacian induced on~$\Sigma$.  We shall
call this quantity the {\it Willmore invariant}; it is invariant under
M\"obius transformations of the ambient~$\mathbb{E}^3$. A~key feature
is the linearity of the highest order term,~$\bar{\Delta} H$.  This is
important for PDE problems, but also means that the Willmore invariant
should be viewed as a fundamental conformal curvature quantity.  The
Willmore equation and its corresponding Willmore energy functional
play an important {\it r\^ole} in both mathematics and physics (see
{\it e.g.}~\cite{Riviere,Polyakov,Alexakis}). Recently the celebrated
Willmore conjecture~\cite{Willmore} concerning absolute minimizers of
this energy was settled in~\cite{Marques}.  The Willmore energy is
also linked to a holographic notion of physical
observables~\cite{GrahamWitten} and in particular entanglement entropy
~\cite{RyuT1,RyuT2,Gibbons,Us}.  It is clearly important to understand
whether the existence of the Willmore invariant is peculiar to
surfaces, or if there also exist higher dimensional analogs.  We
provide the answer to this question: via a natural PDE problem, we
show the Willmore invariant is the first member of a family of
fundamental curvatures existing in each even hypersurface
dimension. The nature of these, and the way that they arise, strongly
suggests that they will also play an important {\it r\^ole} in both
mathematics and physics.  In addition, we shall find related objects for odd
dimensional hypersurfaces, and there is evidence that these should
also be of interest~\cite{GGHW}.

A powerful approach to the study of conformal geometry is provided by
the Poincar\'e metric of Fefferman--Graham~\cite{FGast}, (which in
part follows the approach to CR geometry developed in \cite{F1} and
\cite{ChengYau}).  This is a construction to study the conformal
geometry of a~$(d-1)$-manifold~$\Sigma$ by recovering it as the
boundary of a conformally compact Einstein~$d$-manifold.  The
Poincar\'e metric of $d$-dimensions is not directly suitable for
capturing general extrinsic geometry of $(d-1)$-dimensional
submanifolds, because the Einstein equation significantly constrains
the allowable conformal geometry of the manifold hosting the
hypersurface. For the boundary, only minimal regularity forces its
embedding to be {\em totally umbilic}~\cite{LeB-Heaven,Goal}, {\it
  i.e.}, everywhere vanishing trace-free second fundamental form. On
the other hand, certain interesting conformal submanifold geometry is
revealed by embedding suitable submanifolds in the boundary of a
Poincar\'e manifold~\cite{GrahamWitten}.

Our approach is to replace the Poincar\'e metric problem with one
directly adapted to the situation of a conformally embedded
hypersurface.  We show that the Loewner--Nirenberg-type problem of
finding, conformally, a negative scalar curvature, conformally compact
metric provides this natural replacement. 
This is a non-compact analog of the well-known Yamabe problem. Our use of this problem is in part inspired by
the~1992 article~\cite{ACF} of Andersson, Chru\'sciel and Friedrich
which (building on the works~\cite{AMO1,AMO2,AMO}) identified a
conformal surface invariant that obstructs smooth boundary asymptotics
for the Loewner-Nirenberg type problem (and also gave some information on
the corresponding obstruction in higher dimensions).
In fact, the
 invariant found in~\cite{ACF} is the same as that arising from the
variation of the Willmore energy; in particular its specialisation to
surfaces in~$\mathbb{E}^3$ agrees with the Willmore invariant appearing in~(\ref{Wore}). For later reference, let us now set up this  problem.

\subsection{A scalar curvature boundary problem} \label{prob} 
Given a $d$-dimensional Riemannian manifold~$(M,g)$ with boundary~$\Sigma:=\partial M$, one may ask whether there is a smooth real-valued
function~$u$ on~$M$ satisfying the following  conditions: 
\begin{enumerate}
\item~$u~$ is a defining function for~$\Sigma$ ({\it {\it i.e.}},~$\Sigma$
  is the zero set of~$u$, and~$\boldsymbol{d} u_x\neq 0$~$\forall x\in \Sigma$);
\item\label{(2)}~$\bar{g}:=u^{-2}g$ has scalar curvature
${\rm Sc}^{\bar{g}}=-d(d-1)$. 
\end{enumerate}
Here~$\boldsymbol{d}$ is the exterior derivative (on functions). We
assume dimension~$d\geq 3$ and all structures are~$C^\infty$.

Assuming~$u>0$ and setting~$u=\rho^{-2/(d-2)}$, part (2) of this
problem is governed by the Yamabe equation
\begin{equation}\label{Yamabe}\Big[\!-4\, \frac{d-1}{d-2}\, \Delta+\Sc\Big]\rho +  d\, (d-1)\,  \rho^{\frac{d+2}{d-2}}=0\, .\end{equation}
The above problem fits nicely into the
framework of conformal geometry as follows: Recall that a conformal structure~$\boldsymbol{c}$
on a manifold is an equivalence class of metrics where the equivalence
relation~$\widehat{g}\sim g$ means that~$\widehat{g}= \Omega^2 g$ for
some positive function~$\Omega$. The line bundle~$(\Lambda^d TM)^2$ is
oriented and, for~$w\in \mathbb{R}$, the bundle of {\it conformal
  densities} of weight~$w$ is denoted~$\ce M[w]$ (or simply $\ce[w]$ if the underlying manifold is clear from context), and is defined to be the
oriented~$\frac{w}{2d}$-root of this. Locally each~$g \in \cc$
determines a volume form and, squaring this, globally a section of
$(\Lambda^d T^*M)^2$. So, on a conformal manifold~$(M,\boldsymbol{c})$ there is a
canonical section~$\bg$ of~$\odot^2T^* M\otimes \ce[2]$ called the
conformal metric. Thus each metric~$g \in \cc$ is naturally in~$1:1$
correspondence with a (strictly) positive section~$\tau$ of~$\ce[1]$
via~$g=\tau^{-2} \bg$. Also, the Levi-Civita connection~$\nabla$ of
$g$ preserves~$\tau$, and hence~$\bg$.  Thus we are led from Equation~\nn{Yamabe} to the
conformally invariant equation on a weight~1 density~$\sigma\in
\Gamma(\ce[1])$
\begin{equation}\label{Ytwo}
S(\sigma):= 
\big(\nabla \si \big)^2 - \frac{2}{d} \si\, \Big(\Delta +\frac{\rm Sc}{2(d-1)}\Big) \si = 1 , 
\end{equation}
where~$\bg$ and its inverse are used to raise and lower indices, 
$\Delta = \bg^{ab}\nabla_a\nabla_b$
and ~${\rm Sc}$ means~$\bg^{bd}R_{ab}{}^a{}_d$, with~$R$ the
Riemann tensor. Now, choosing~$\boldsymbol{c}\ni g=\tau^{-2} \bg$, 
Equation~(\ref{Ytwo}) is the
 PDE governing the function~$u=\si/\tau$ solving part~(\ref{(2)}) of the problem. Setting aside boundary aspects, this is exactly the Yamabe equation~\eqref{Yamabe} above. 

The critical point is that, in contrast to~\eqref{Yamabe},
Equation~\eqref{Ytwo} is well-adapted to the boundary problem since $u$ (or equivalently $\sigma$) is the variable defining the boundary.
Let us write $\Sigma:=\Z(\sigma)$ for the zero locus of $\sigma$.
Since~$u$ is a defining function, this implies that~$\si$ is a {\em defining density} for~$\Sigma$, meaning that it is a
section of~$\ce[1]$ with zero locus~$\Sigma$  and~$\nabla \si_x\neq 0$~$\forall x\in \Sigma$.  For our purposes, we only
need to treat the problem formally (so it applies to any
hypersurface) and it may thus be stated as follows:

\begin{problem}\label{conffirststep} Let~$\Sigma$ be an embedded
hypersurface in a conformal manifold~$(M,\boldsymbol{c})$ with dimension~$d\geq 3$.
Find a smooth defining density~$\bar \sigma$ such that
\begin{equation*}
S(\bar{\si})=1 + \bar{\sigma}^{\ell} A_\ell\, ,
\end{equation*}
for some~$A_\ell\in \Gamma(\ce[-\ell])$,
where~$\ell \in\mathbb{N}\cup\infty$ is as high as possible. 
\end{problem}
\noindent Here and elsewhere $\Gamma({\mathcal V})$ indicates the space
of smooth, meaning $C^\infty$, sections of a given bundle
$\mathcal{V}$.

\subsection{The main results}\label{results}
Our first main result is an asymptotic solution to the singular Yamabe
Problem~\ref{conffirststep}. This is given in Theorem~\ref{obstr},
which states that in solving the equivalent Equation\nn{Ytwo} we can smoothly achieve
\begin{equation}\label{Bdef}
S(\si)=1+\si^d B \quad \mbox{equivalently}\quad \Sc^{\si^{-2}\boldsymbol{g}}=-d(d-1)+ \si^d C,
\end{equation}
for a certainly smooth conformal density $B$ (and $C=-d(d-1)B$).
Proposition~\ref{prod} and Proposition~\ref{log}  provide the
explict recursive formulae that solve the problem, including log terms.
These simple formul\ae\ follow after recasting Problem~\ref{conffirststep}  in terms of tractor calculus on $(M,\boldsymbol{c})$; in
particular as the almost scalar constant (ASC) problem stated
in~\cite{Goal}.  The equivalent tractor problem is given in
Problem~\ref{I2-prob}, and provides an effective simplification.
The
approach demonstrates that surprisngly, although the problem is non-linear,
 a fundamental $\frak{sl}(2)$ structure inherent in the
geometry \cite{GW} (see Section \ref{sl2structure}) may be applied.

The next main result concerns the obstruction to smoothly solving
Problem~\ref{conffirststep} beyond order $\ell=d$:
\begin{theorem}\label{main-Th-var} There is an obstruction to  solving
Problem~\ref{conffirststep} smoothly to all orders and this 
 is a conformal density $\cB\in \Gamma(\ce\Sigma[-d])$ determined
entirely by the data $(M,\cc,\Sigma)$, of the conformal embedding of
the boundary.
That is a smooth
formal solution to the singular Yamabe Problem~\ref{conffirststep}
exists if and only if the hypersurface invariant $\cB$ is zero.
\end{theorem}
\noindent This theorem is proved as the equivalent statement
Theorem~\ref{main-Th}.  From that we see that the {\it (ASC)
  obstruction density} $\cB$ arises as $B|_\Sigma$, where $B$ is as in
\nn{Bdef}.  We later prove that the density~$\cB$ is natural in the
sense that, in a scale, it may be given by an expression polynomial in
the hypersurface conormal, ambient Riemann curvature and Levi-Civita
covariant derivatives thereof, see Theorem \ref{Bnatthm}. 
So it is conformal invariant of the hypersurface; hypersurface invariants are defined in Definitions \ref{R-invtdef} and \ref{chi-def} of Sections \ref{definv-sec} and \ref{cyp}. If~$\cB$ is zero for a given
hypersurface, then any smooth formal solution to all orders
depends on the choice of a smooth conformal density
in~$\Gamma(\ce\Sigma[-d])$.

Theorem~\ref{leading} establishes that for 
hypersurfaces of even dimension $d-1$ the obstruction density takes the form 
$$ 
\bar{\Delta}^{\frac{d-1}{2}} H + \mbox{\rm lower order terms},
$$ 
generalising the Willmore invariant $\bar{\Delta}H +\mbox{\rm lower
  order terms} $, {\it cf.} Equation \nn{Wore}. Then we prove the following (see page \pageref{ui}): 
\begin{theorem} \label{uniqui} 
The obstruction density is a fundamental curvature invariant of 
even dimensional conformal hypersurfaces.
For hypersurfaces of dimension 3 or greater, the trace-free second
fundamental form $\IIo$ and the obstruction density $\cB$ are the only
fundamental conformal invariants of a hypersurface (up to the addition
of lower order terms) taking values in an irreducible bundle.
\end{theorem}
\noindent The notion of {\em fundamental curvature quantity} used in
the Theorem here is defined in Section \ref{BGG}.  Informally it means a
conformal invariant that has a non-trivial linearisation with respect
to embedding variation of flat structures, and in general cannot be
constructed in a standard way from simpler invariants, see
Remark~\ref{properties}.  This property and the leading order
behaviour show that on even dimensional hypersurfaces,~$\cB$
generalises the Willmore invariant.  We also show in Section~\ref{BGG}
that the obstruction density is a scalar density analog of the
Fefferman--Graham obstruction tensor~\cite{FGrnew}.  This uses the
classification of linear conformally invariant operators and the
appropriate Bernstein--Gel'fand--Gel'fand (BGG) complexes.
Proposition~\ref{Willmore}  gives
explicit formul\ae\ for the obstruction density for hypersurfaces of
dimensions two and three (with~$(M,\cc)$ conformally flat for the
latter case). In two dimensions this gives exactly the Willmore
invariant.

The uniqueness  of the
solution~$\bar\si$ to Problem~\ref{conffirststep} means that, in a
conformal manifold~$(M,\cc)$, a hypersurface~$\Sigma$ determines a
corresponding distinguished {\em conformal unit defining density}, as
defined in Definition~\ref{cunit}; this is uniquely determined up to
the order asserted in Theorem \ref{obstr}.  This gives an interesting
way to construct conformal hypersurface invariants: Consider any
conformal invariant~$U$ on~$M$, that couples the data of the jets of
the conformal structure~$(M,\cc)$ to the jets of the
section~$\bar\si$.  Along~$\Sigma$, this determines an invariant
of~$(M,\cc,\Sigma)$ whenever~$U$ involves there no more than
the~$d$-jet of~$\bar\si$. This is treated in Section~\ref{invts}.

 Importantly, the stated uniqueness of~$\bar \sigma$ also means that
 we can construct the basic invariant differential operators
 determined naturally by the data $(M,\cc,\Sigma)$.  In particular we
 give a family of conformally invariant, hypersurface Laplacian-power
 type operators depending on the extrinsic geometry and termed {\it
   extrinsic conformal Laplacians}:
\begin{theorem}\label{ecL-var} Consider a hypersurface $\Sigma$ embedded in a conformal manifold $(M,\cc)$. The data of this embedding determines canonically, 
for each $k\in 2\mathbb{Z}_{>0}$,  a natural
 conformally invariant differential  operator
$$\Pop_k:\Gamma\big(\ct^\Phi M\big[\frac{k-d+1}{2}\big]\big)\Big|_\Sigma\rightarrow \Gamma\big(\ct^\Phi M\big[\frac{-k-d+1}{2}\big]\big)\Big|_\Sigma\, ,$$
with leading term $\bar{\Delta}^{\frac{k}{2}}$.  
\end{theorem}
\noindent This is proved in Theorem \ref{ecL} and Theorem
\ref{alllaps}. In fact, as we indicate in the first of these, certain
linear operators also arise for $k$ odd.  In one sense the $ \Pop_k$
are analogs of the Graham--Jennes--Mason--Sparling (GJMS)
operators~\cite{GJMS}, but in contrast exist to all orders in both
dimension parities and depend non-trivially on the conformal
embedding.  These operators are based on the tangential operators
found in~\cite{GW}, and admit simple holographic formul\ae--see
Theorem~\ref{ecL}.  
%Proposition~\ref{GJMS23} gives explicit tensorial
%formul\ae\ for embedded surfaces and volumes.
 An important application of the extrinsic conformal Laplacians is to the
 provision of a (holographic) formula for the obstruction
 density. This is given in Theorem~\ref{holoB}.

The first draft of this article also studied energy
functions for the obstruction density and, based on evidence gathered at linear order,  posed the question of whether
such action functionals exist in all dimensions 
(a detailed discussion of integrated invariants may now be found in~\cite{GW-willII,volume},
explicit formul\ae\ in ambient dimensions 3 and 4 and a novel variational calculus 
are given in~\cite{GGHW}).
Since then,
Graham  has answered this question by showing that in all
dimensions the obstruction density we define arises as the gradient,
with respect to variations of the conformal embedding, of an integral that
gives the anomaly term in a renormalised volume expansion for singular
Yamabe metrics solving Problem~\ref{conffirststep} below, see
\cite{Gra}. This conclusion was subsequently also recovered as
a special case of a broad framework for renormalised volume problems developed in \cite{volume}, where
it is moreover shown that the given anomaly/energy is an integral of
an extrinsically coupled $Q$-curvature.

\subsection{Structure of the article}\label{str}

In Section~\ref{Riemann} below, we show that the classical problem of
constructing Riemannian hypersurface invariants can be treated
holographically. This provides the idea to be generalised in the more
difficult conformal case, and also a tool for calculating in the
treatment of the latter.  A review of basic conformal geometry and
tractor calculus is given in Section~\ref{ASC-sec}, with key
identities derived in Section~\ref{idents}.  Conformal hypersurfaces
are treated in Section 4.  Section~\ref{lin} deals with the linearised
obstruction density.

%%  In
%% Section~\ref{comp-ext} we develop theory and tools for computing 
%%   conformal hypersurface invariants.
%% These are built and  determined canonically from the jets of the conformal unit defining density, up to the order of the obstruction.
Theorem \ref{obstr} allows us proliferate natural invariants of
conformal hypersurfaces; this is discussed in Section~\ref{invts}
where various explicit examples are also given.  A recursive procedure
for computing conformal hypersurface invariants is given in
Theorem~\ref{main-calc}.  

%% Combining the tractor calculus machinery for the ambient space and the
%% conformal unit defining density, yields a conformal hypersurface
%% calculus. Its key ingredients are presented in Section~\ref{hollow}.
%% These include the tractor analog of the second fundamental form, a
%% hypersurface Thomas D-operator and the tractor connection analog of
%% the Riemannian geometry Gau\ss\ formula.  Sections~\ref{invops}
%% and~\ref{critfs} develop the extrinsic conformal Laplacians and
%% related action functionals.

\subsection{Notation}
We will primarily employ an abstract index notation in the spirit of~\cite{ot}. Occasionally a mixed notation such as $R(\hat n,b,c,\hat n)=R_{abcd} \hat n^a \hat n^d$, where $\hat n$ is a vector field and $b,c$ are indices, is propitious.

\addtocontents{toc}{\SkipTocEntry}
\subsection*{Acknowledgements}
Prior to this work, A.R.G. had discussions related to this problem
with F.\ Marques and then P.\ Albin and R.\ Mazzeo. We are indebted
for the insights so gained. Both authors would also like to thank
C.R. Graham and Y. Vyatkin for helpful comments.  A.W. thanks
R. Bonezzi, M. Halbasch, M. Glaros for discussions.  The authors
gratefully acknowledge support from the Royal Society of New Zealand
via Marsden Grant 13-UOA-018 and the UCMEXUS-CONACYT grant CN-12-564.
A.W. thanks for their warm hospitality the University of Auckland and
the Harvard University Center for the Fundamental Laws of
Nature. A.W. was also supported by a Simons Foundation Collaboration
Grant for Mathematicians ID 317562.

\section{Riemannian hypersurfaces}\label{Riemann}

Throughout this section we shall consider a smooth ($C^\infty$) {\it hypersurface}
$\Sigma$ in a Riemannian manifold~$(M^d,g)$, meaning a smoothly embedded codimension~1 submanifold. Treating Riemannian hypersurfaces  has several purposes.
It enables us to set up the basic structures for later use. We also use this  opportunity to 
 show that
such a hypersurface determines a canonical defining function. This provides  a method for constructing Riemannian hypersurface  invariants and 
 thus  gives  a  simple analog of the conformal
ideas studied in the subsequent sections.

\subsection{Riemannian notation and conventions}\label{R-conv}

We work on manifolds~$M$ of dimension~$d\geq 3$, unless stated
otherwise. For simplicity we assume that this is connected and
orientable. Also for simplicity we will consider only metrics of
Riemannian signature, however nearly all results can be trivially
extended to metrics of any signature~$(p,q)$.  Given a metric $g$ we
write $\nabla_a$ to denote the corresponding Levi-Civita connection.
Then the Riemann curvature tensor~$R$ is given~by
\[
R(u,v)w=\nd_{u}\nd_{v}w-\nd_{v}\nd_{u}w-\nd_{[u,v]}w,
\]
where~$u$,~$v$, and~$w$ are arbitrary vector fields. In an abstract
index notation ({\it cf.}~\cite{ot})~$R$ is denoted by~$R_{ab}{}^{c}{}_d$,
and~$R(u,v)w$ is~$u^av^bw^d R_{ab}{}^{c}{}_d$.
This can be decomposed into the totally trace-free {\em Weyl curvature}
$W_{abcd}$ and the symmetric {\em
Schouten tensor}~$\Rho_{ab}$ according~to
\begin{equation}\label{Rsplit}
R_{abcd}=W_{abcd}+2g_{c[a}\Rho_{b]d}+2g_{d[b}\Rho_{a]c},
\end{equation}
where~$[\cdots]$ indicates  antisymmetrisation over the enclosed
indices. 
Thus~$\Rho_{ab}$ is a trace modification of the Ricci tensor 
${\rm Ric}_{ab}=R_{ca}{}^c{}_b$:
$$
\Ric_{ab}=(d-2)\Rho_{ab}+ \J g_{ab}\, , \quad \quad \J:=\Rho^a_{a}\, .
$$ 
The scalar curvature is~$\Sc=g^{ab}\Ric_{ab}$ so~$\J=\Sc/(2(d-1))$;
in two dimensions we define ~$\J:=\frac12\Sc$. To simplify notation, we will often write~$u.v$ to  denote $g(u,v)$, for vectors~$u$ and $v$. 
%Finally, the Einstein tensor~$G$ is given by~$G=\Ric-\frac 12 g \Sc$.

\subsection{Hypersurfaces in Riemannian manifolds}\label{R-hyp}
Given a hypersurface~$\Sigma$, a function $s\in C^\infty M$ will be called a {\it defining function} for $\Sigma$ if $\Sigma=\Z(s)$
(the zero locus of $s$) and the exterior derivative $\dd s$ 
is nowhere vanishing along~$\Sigma$.
Defining functions always exist, at least locally.
We shall use the notation~$n_{s}$ (or simply~$n$ if~$s$ is understood)
for~$\dd s$.
\begin{definition}
A defining function~$s$ for~$\Sigma$ is said to be {\em normal}\, if 
\begin{equation}\label{def-norm}
|n_s|^2=1+s A\, ,
\end{equation} 
where~$A\in C^\infty(M)$.
\end{definition}
To see that normal defining functions always exist locally, let~$s$ be
a defining function  for a
smooth hypersurface~$\Sigma$ in a Riemannian manifold~$(M^d,g)$. Then
there is a neighbourhood~$U$ of~$\Sigma$ on which the function~${\bar
  s}:=s/|\dd s|_g$ is a defining function, and this is normal.

Now we consider further ``improving'' the defining function:
\begin{problem}\label{Riemannsfirststep}
Given~$\Sigma$, a smooth hypersurface in a Riemannian manifold~$(M,g)$
find a defining function~$\bar s$ such that~$n_{\bar s}$ obeys
\begin{equation}\label{ellp1}
|n_{\bar s}|^2=1+ s^{\ell+1} A\, ,
\end{equation}
for some~$A\in C^\infty(M)$ and~$\ell\in\Nint\cup\infty$ as high as possible.
\end{problem}

\begin{remark} In fact the equation~$|n_{\bar s}|=1$ may be solved exactly in a 
neighbourhood of~$\Sigma$ (this is  part of the construction of Gaussian normal coordinates~\cite{Wald} in which $\bar s$ is the distance to the hypersurface) but we wish to
illustrate an inductive approach to the problem as stated. The point
being that an adaptation of this idea then treats  Problem~\ref{conffirststep}.  Furthermore, the canonical defining function
obtained by this process is very useful for calculations,
{\it cf.}~\cite{GW-willII}. 
%\end remark
\end{remark}

Next some points of notation. Equalities such as~$E=F+s^k G$ for~$G$
smooth, will be denoted~$E=F+\O({s^k})$, and we may write $E\stackrel\Sigma = F$ if $k\in {\mathbb Z}_{\geq 1}$. 
We will also use this notation when either $E$ or $F$ is defined only along~$\Sigma$ and agrees with the restriction of the other side to~$\Sigma$.
Norms squared of vector or 1-form fields will often be denoted simply by squares, {\it i.e.},  
$v^2:=|v|^2$. Hence our problem is given~$s$ such that~$(\dd
s)^2=1+\O(s)$, find~$\bar s(s)$ such that~$(\dd \bar
s)^2=1+\O(s^{\ell+1})$.

Now the key point here is that Problem~\ref{Riemannsfirststep} can 
be solved, to arbitrarily high order by an explicit recursive formula. 
This is based on the following Lemma.
\begin{lemma}\label{recur}
Suppose the  defining function~$s\in C^\infty (M)$ satisfies
$$
n^2=1+\O(s^{\ell})\, , \quad\ell\in\Zint_{\geq 1},
$$
where~$n:=\dd s$.
Then 
\begin{equation}\label{sbar}
\bar s=s\,  \Big[1-\frac 1{2}\frac{n^2-1}{\ell+1}\Big]
\end{equation}
solves~\nn{ellp1}.
\end{lemma}
\begin{proof}
Since~$n^2=1+s^{\ell} A$, with~$\ell\geq 1$, for some~$A\in C^\infty(M)$, then 
$$
\nabla \bar s = n - \frac{1}{2(\ell+1)}\big((\ell+1) s^{\ell}  A\,  n+ s^{\ell+1} \nabla A\big)=\big(1-\frac1{2} s^\ell A \big)\, n+\O(s^{\ell+1})\, .
$$
Hence
$$
(\nabla \bar s)^2= \big(1 - s^\ell  A)\, n^2 + \O(s^{\ell+1}) = 1+\O(s^{\ell+1})\, ,
$$
where the last equality used the fact that~$n^2 = 1 + s^\ell A$. Note that, at the order treated, Equation~\nn{sbar} uniquely specifies the adjustment of $\bar s$ required  to solve~\eqref{ellp1}.
\end{proof}

An immediate consequence of this Lemma is a recursive formula solving
Problem~\ref{Riemannsfirststep}.

\begin{proposition}\label{R-main}
Problem~\ref{Riemannsfirststep} can be solved uniquely to order~$\ell=\infty$. 
The solution is given recursively by
$$
\bar s:= \bar s_\ell
$$
where
$$
\bar s_\ell= s \, \prod_{k=0}^{\ell-1}  \Big[1-\frac 1{2}\frac{n_{\bar s_k}^{\, 2}-1}{k+2}\Big]\, ,
$$
and~$\bar s_0:=s$ is a normal defining function.
\end{proposition}

We will term a defining function~$s$ obeying~$(\nabla s)^2 = 1$
everywhere a {\it unit defining function}.  In the setting of formal
asymptotics, we will slightly abuse this language by using it also to
refer to solutions to Problem~\ref{Riemannsfirststep} (whose existence
is guaranteed by the previous proposition).

\subsection{The standard Riemannian hypersurface objects and identities}\label{standard}
Let~$\Sigma$ be a smooth hypersurface in a Riemannian manifold~$(M^d,g)$, and write~$\nabla~$ for the Levi-Civita connection of~$g$.
Suppose that~$\hat n^a\in \Gamma (TM)$ is such that its restriction to~$\Sigma$ is a unit normal vector field. The tangent bundle~$T\Sigma$
and the subbundle of~$TM|_\Sigma$ orthogonal to~$\hat n^a$ along~$\Sigma$ (denoted~$TM^\top$) may be identified; this will be assumed
in the following discussion. Thus we will use the same abstract
indices for~$T\Sigma=TM^\top$ as for~$TM$.  

We will also denote the projection of tensors to submanifold tensors
by a superscript~$\top$. In particular, for a vector~$v\in \Gamma (TM)$,
we have~$v^\top:=v-\hat n\, \hat n.v$.  We will generally use a bar to distinguish intrinsic hypersurface quantities from the corresponding ambient objects.
 In particular, for a vector field~$v^a\in \Gamma (T\Sigma)$,
the intrinsic Levi--Civita connection~$\bar \nabla$ is given in terms
of the ambient connection restricted to~$\Sigma$ by the Gau\ss\ formula, 
\begin{equation}\label{hypgrad}
\nablab_a v^b = \nabla_a^\top  v^b+\hat n^b\II_{ac} v^c\, ,
\end{equation}
where~$\nabla_a^\top$ means~$(\delta^b_a-\hat{n}^b\hat{n}_a)\nabla_b$ and
the second fundamental form
$\II_{ab}\in \Gamma \big(\!\odot^2T^{*}\Sigma\big)$ is given~by
\begin{equation}\label{two}
\II_{ab}=\nabla_a^\top \hat n_b\big|_\Sigma\, .
\end{equation}
Its trace yields the mean curvature
%~$H$
$$
H:=\frac1{d-1}\,  \II^a_a\, ,
$$
so that the trace-free second fundamental form~$\IIo_{ab}$ can be written
$$
\IIo_{ab}=\II_{ab}-H \bar g_{ab}\, ,
$$
where (along~$\Sigma$) $g_{ab}-\hat n_a \hat n_b$ is called  first fundamental form or  the induced metric $\bar g_{ab}$.
The intrinsic curvature is related to the (projected) ambient curvature along~$\Sigma$ by the Gau\ss\  equation 
$$
\bar R_{abcd}=R^\top_{abcd}+\II_{ac}\II_{bd}-\II_{ad}\II_{bc}\, .
$$
When~$d\geq4$, writing the trace of this expression  in terms of the trace-free second fundamental form, as well as Schouten and Weyl tensors, leads to
what we shall call the Fialkow--Gau\ss\ equation:
\begin{equation}\label{Gausstrace}
\IIo^2_{ab}\!-\frac12\, \bar g_{ab}\scalebox{1.2}{$\frac{\IIo_{cd_{\phantom{a}\!\!}}\!\IIo^{cd}}{d-2}$}-W(\hat n,a,b,\hat n)\!=\!(d-3)\!\Big(\!\Rho_{ab}^\top-\bar \Rho_{ab}+H\IIo_{ab}+\frac 12\,  \bar g_{ab}H^2\!\Big)\!=:\!(d-3)\F_{ab}\, .
\end{equation}
Since~$W(\hat n,a,b,\hat n)$ ($=:W_{cabd}\hat n^c\hat n^d$) and~$\IIo_{ab}$ are both known to be conformally invariant,  the same applies to the expression on the right hand side of the above display. This was coined
the Fialkow tensor in~\cite{YuriThesis}.

Note that in Equation~\nn{Gausstrace}, some terms such as $W(\hat n,a,b,\hat n)$ may be defined off~$\Sigma$, but it is implicitly clear that this expression as whole only makes sense along~$\Sigma$. In such situations, we  avoid cumbersome notations such as $W(\hat n,a,b,
\hat n)\big|_\Sigma$ wherever clarity allows. 

The covariant curl of the second fundamental form is governed by the Codazzi--Mainardi equation
\begin{equation}\label{Mainardi}
\nablab_a\II_{bc} - \nablab_b \II_{ac}= \big(R_{abcd}\hat n^d\big)^{\!\top}\, ;
\end{equation}
tracing this yields (in~$d\geq 3$)
\begin{equation}\label{Mainarditrace}
\nablab.\IIo_b-(d-2)\nablab_b H=(d-2)\Rho(b,n)^\top\, .
\end{equation}
Conversely, the trace-free part of the Codazzi--Mainardi
equation gives
\begin{equation}\label{CODIIO}
\nablab_a \IIo_{bc}-\nablab_b\IIo_{ac}+\frac1{d-2}\big(\nablab.\IIo_a\,  \bar g_{bc}
-\nablab.\IIo_b\,  \bar g_{ac}\big)
=\big(W_{abcd}\hat n^d)^\top\, .
\end{equation}
Since the Weyl tensor is conformally invariant, this shows that the trace-free curl of the trace-free second fundamental form is also invariant. Indeed the 
mapping of weight~1, trace-free, rank two symmetric tensors 
\begin{equation}
\label{CodazziOperator}
\mathring K_{bc}\xmapsto{\sf Cod} 
\nablab_a \mathring K_{bc}-\nablab_b \mathring  K_{ac}+\frac1{d-2}\big(\nablab.\mathring K_a\,  \bar g_{bc}
-\nablab.\mathring K_b\,  \bar g_{ac}\big)\, ,
\end{equation} is itself a conformally invariant operator. The operator ${\sf Cod}$ will be called the {\it conformal Codazzi operator} in the following.

The link between the intrinsic and ambient scalar curvatures is given by the formula
$$
\II_{ab} \II^{ab} - (d-1)^2 H^2 =  \Sc -2\Ric(\hat n,\hat n)-\overline{\Sc}\, .
$$
This recovers  Gau\ss' {\it Theorema Egregium} for  a three dimensional Euclidean ambient space.
In dimension~$d>2$, the above display may be written in terms of the Schouten tensor and trace-free second
fundamental form and gives
\begin{equation}\label{JJbar}
\J- \Rho(\hat n,\hat n)=\bar\J -\frac{d-1}{2}\, H^2 + \frac{\IIo_{ab}\IIo^{ab}}{2(d-2)}\, .
\end{equation}
This result shows that~$\J-\Rho(\hat n,\hat n)-\bar \J +\frac{d-1}{2}\, H^2$ is conformally invariant.

Finally, we record the relation between the Laplacian of the mean curvature and 
divergence of the second fundamental form; this is a simple consequence of the Codazzi--Mainardi equation:
\begin{equation}\label{boxH}
\begin{split}
\bar\Delta H&=
\frac1{d-1}\nablab^a\big(\nablab.\II_a-\Ric(\hat n,a)^\top\big) =\frac{1}{d-2}\nablab^a\big(\nablab.\IIo_a-\Ric(\hat n,a)^\top\big)\, .
\end{split}
\end{equation}

\subsection{Riemannian hypersurface invariants}\label{definv-sec}
Since locally any hypersurface is the zero set of some defining function, there is no loss of generality in restricting to
those hypersurfaces~$\Sigma$ which are the zero locus~$\mathcal{Z}(s)$
of some defining function~$s$. To further simplify our discussion  we
also assume that~$M$ is oriented with volume form
$\omega$. 
Given a hypersurface in~$M$, it has an orientation
determined by~$s$ and~$\omega$, as~$\boldsymbol{d} s$ is a conormal
field. Different defining functions are {\em compatibly oriented} if
they determine the same orientation on~$\Sigma$.

\begin{definition}\label{R-invtdef}
For hypersurfaces, a {\em scalar Riemannian pre-invariant}  is a function
$P$ which assigns to each pair consisting of a Riemannian~$d$-manifold
$(M,g)$ and hypersurface defining function~$s$, a function
$P(s ; g)$ such that: \\ 
\begin{enumerate}
\item[(i)] $P(s; g)$ is natural, in the sense that
for any diffeomorphism~$\phi:M\to M$ we have~$P(\phi^* s;\phi^* g ) =
\phi^* P(s; g)$.\\ 
\item[(ii)] The restriction of~$P(s; g)$ is independent of the choice of
oriented defining functions, meaning that if~$s$ and~$s'$ are two
compatibly oriented defining functions such that
$\mathcal{Z}(s)=\mathcal{Z}(s')=:\Sigma$ then,~$P(s; g)|_\Sigma= P(s';
g)|_\Sigma$.\\
\item[(iii)] $P$ is given by a universal polynomial expression such that, 
given a local coordinate system~$(x^a)$ on~$(M,g)$,~$P(s; g)$ is given by
a polynomial in the variables
$$g_{ab},~\partial_{a_1}g_{bc}, ~\cdots, ~\partial_{a_1}\partial_{a_2}\cdots \partial_{a_k}
g_{bc},~(\det g)^{-1},$$
$$s,~\partial_{b_1} s,~\cdots~,\partial_{b_1}\partial_{b_2}\cdots \partial_{b_\ell} s, 
~|| \boldsymbol{d} s ||_g^{-1}, \omega_{a_1\ldots a_d}\, ,$$ for some positive integers~$k,\ell$. 
\end{enumerate}
A {\em scalar Riemannian invariant} of a hypersurface~$\Sigma$ is the
restriction~$P(\Sigma;g):=P(s;g)|_\Sigma$ of a 
pre-invariant~$P(s;g)$ to~$\Sigma:= \mathcal{Z}(s)$.
\end{definition}
In (iii)~$\partial_a$ means~$\partial/\partial x^a$,
$g_{ab}=g(\partial_a,\partial_b)$,~$\det g= \det(g_{ab})$ and $\omega_{a_1\ldots a_d}=\omega(\partial_{a_1},\ldots,\partial_{a_d})$.  
Also, in the context of treating  hypersurface invariants, there is no loss of generality studying defining functions such that  $|| \boldsymbol{d} s ||_g^{-1}\neq 0$ 
everywhere in $M$, since we may, if necessary replace $M$ by a local neighborhood of $\Sigma$.
For
(i) note that if~$\Sigma=\mathcal{Z}(s)$, then~$\phi^{-1} (\Sigma)$
is a hypersurface with defining function~$\phi^* s$. The conditions (i),(ii) and
(iii) mean that any Riemannian invariant~$P(s;g)|_\Sigma$ of~$\Sigma$, is
entirely determined by the data~$(M,g,\Sigma)$, and this justifies the notation~$P(\Sigma;g)$. Then in this notation the naturality condition of~(i) implies
~$\phi^*\!( P(\Sigma,g))= P(\phi^{-1}(\Sigma),\phi^* g)$.

While scalar invariants are our main focus, the above
definition is easily extended to define {\em tensor valued
  hypersurface pre-invariants} and {\em invariants}. In that case one
considers instead tensor valued functions~$P$ and requires that the
coordinate components of the image satisfy the conditions (ii), (iii), and the obvious adjustment of (i).  We shall use the term {\em invariant} to mean either
tensor or scalar valued hypersurface invariants.

A simple example of such a tensor valued invariant is the unit
conormal defined by a defining function
$s$, that is ~$\hat n:= \boldsymbol{d} s/|| \boldsymbol{d} s ||_g$. Then the second
fundamental form~$\II$ and its intrinsic covariant derivatives
$\nablab \II$,~$\nablab \nablab \II$ and so forth are easily seen to
arise from the restriction to~$\Sigma$ of tensor-valued
pre-invariants. So, for example,~$\II^{ab}\II_{ab}$ is a scalar hypersurface 
invariant.

It follows that, in the spirit of Weyl's clasical invariant theory, one may generate Riemannian hypersurface invariants by expressions of the form
\begin{equation}\label{parcon}
\mbox{Partial-contraction}\big(n\!\cdots\! n\,  (\nablab \!\cdots\! \nablab \II )
\ldots (\nablab \!\cdots\! \nablab \II ) (\nabla \!\cdots\! \nabla R ) \ldots
(\nabla \!\cdots\! \nabla R ) \big) \big|_\Sigma\, ,\end{equation} where
``{Partial-contr\-action'' is some use of the inverse metric (and the
  volume form if~$M$ is oriented) to contract some number of indices
  (and as usual we view $T\Sigma$ and $T^*\Sigma$ as subbundles of,
  respectively, $TM|_\Sigma$ and $T^*M|_\Sigma$).

\subsection{Invariants via the unit defining function}\label{R-viaunit}

Proposition~\ref{R-main} asserts that a Riemiannian hypersurface
embedding $\Sigma \hookrightarrow M$, (formally) determines a unit
defining function~$\bar{s}$ that is unique to $\O({\bar s}^\infty)$
along~$\Sigma$. In fact is straightforward to re-express the
derivatives of $\bar{s}$, along $\Sigma$, in terms of the objects
introduced above:
\begin{proposition}\label{ind4R}
If~$\bar{s}$ is a unit defining function for a Riemannian hypersurface
$\Sigma$ then, for any integer~$k\geq 1$, the quantity 
$$
\nabla^k \bar{s}|_\Sigma
$$ may be expressed as~$\nablab^{k-2}\II$ plus a
linear combination of partial contractions involving $\nablab^{\ell}\II$ for~$ 0\leq \ell\leq k-3$, and the
Riemannian curvature~$R$ and its covariant derivatives (to order at
most~$k-3$) and  the undifferentiated
conormal~$n$,.
\end{proposition}
\noindent Thus via Proposition~\ref{R-main} every partial contraction
$$
\mbox{Partial-contraction}\big( (\nabla \cdots \nabla \bar{s} ) \ldots (\nabla \cdots \nabla \bar{s} )  (\nabla \cdots \nabla R ) \ldots (\nabla \cdots \nabla R ) \big) \big|_\Sigma\, ,
$$ 
encodes a Riemannian hypersurface invariant and may be written in
terms of expressions of the form \nn{parcon} using
Proposition~\ref{ind4R}. In the following we 
establish an analogous approach for conformal hypersurface invariants.

\section{Conformal geometry and the ASC Problem} \label{ASC-sec}

A conformal structure~$\cc$ is an equivalence class of Riemannian
metrics where any two metrics~$g,g'\in \cc$ are related by a {\it conformal rescaling}; that is~$g'=f g$ with $C^\infty M\ni f>0$.
On a conformal manifold~$(M,\cc)$, there is no distinguished connection on
$TM$. However there is a canonical metric~$h$ and linear connection
$\nabla^{\ct}$ (preserving~$h$) on a related higher rank vector bundle known as
the tractor bundle. This enables us to greatly simplify the treatment of  Problem~\ref{conffirststep}.

\subsection{Basic conformal tractor calculus}\label{trac-sec}
We follow here the development
 of~\cite{BEG}, see also \cite{GoPetCMP}. The standard
tractor bundle and its connection are linked and equivalent to
the normal conformal Cartan connection~\cite{CapGoirred,CapGoTAMS}, and are equivalent to objects
developed by Thomas~\cite{Thomas}.

On a conformal~$d$-manifold~$(M,\cc)$, the standard tractor
bundle~$\ct M$ (or simply~$\cT$ when ~$M$ is understood or~$\cT^A$ as its abstract index notation) is a
rank~$d+2$ vector bundle equipped with a canonical
tractor connection~$\nabla^\ct$. The bundle~$\ct$ is not
irreducible but has a  composition series summarised via
a semi-direct sum notation $$ \cT^AM= \ce M[1]\lpl
\ce_a M[1]\lpl\ce M[-1]\, .$$
Here~$\ce M[w]$ (or~$\ce[w]$ when~$M$ is understood), for~$w\in {\mathbb R}$, is the
conformal density bundle which recall, is the natural (oriented) line bundle 
equivalent, via the conformal structure~$\cc$, to~$\big[(\wedge^d TM)^2\big]^{\frac{w}{2d}}$. Then 
  while~$\ce_a M[w]$ denotes~$\ce M[w]\otimes T^*M=:T^*M[w]$.
This means that there canonically is a  bundle
inclusion~$ \ce[-1] \to \cT^A$, with~$\ce_a[1]$ a subbundle of the
quotient by this, and there is a surjective bundle map~$\cT\to
\ce[1]$.  We denote by~$X^A$ the canonical section of~$
\cT^A[1]:=\cT^A\otimes \ce[1]$ giving the first of these:
\begin{equation}\label{X-incl}
X^A:\ce[-1]\to \cT^A\,  ;
\end{equation}
we refer to~$X$ as the {\em canonical tractor}.

A choice of metric~$g \in \cc$ determines an
isomorphism
\begin{equation}\label{split}
\mathcal{T} \stackrel{g}{\cong} \ce[1]\oplus T^*\!M[1]\oplus
\ce[-1] ~.
\end{equation}
We may write, for example,~$U\stackrel{g}{=}(\si,~\mu_a,~\rho)$, or
alternatively~$[U^A]_g=(\si,~\mu_a,~\rho)$, to mean that~$U$ is an invariant
section of~$\ct$ and~$(\si,~\mu_a,~\rho)~$ is its image under the
isomorphism~\nn{split}. Sometimes we will use~\nn{split}  without emphasis on the metric~$g$, if this is
understood by context.
Changing to a conformally related metric
$\widehat{g}=e^{2\Up}g$  gives a different
isomorphism, which is related to the previous one by the transformation
formula
\begin{equation}\label{transf}
\widehat{(\si,\mu_b,\rho)}=(\si,\mu_b+\si\Up_b,\rho-\bg^{cd}\Up_c\mu_d-
\tfrac{1}{2}\si\bg^{cd}\Up_c\Up_d),  
\end{equation}
where~$\Upsilon_b$ is the one-form~$\Omega^{-1}\dd\Omega$.

 In terms of the above {\it splitting}, the tractor connection is given by 
\begin{equation}\label{trconn}
\nd^{\ct}_a
\left( \begin{array}{c}
\si\\[1mm]\mu_b\\[1mm] \rho
\end{array} \right) : =
\left( \begin{array}{c}
    \nabla_a \si-\mu_a \\[1mm]
    \nabla_a \mu_b+ \bg_{ab} \rho +\Rho_{ab}\si \\[1mm]
    \nabla_a \rho - \Rho_{ac}\mu^c  \end{array} \right) .
\end{equation}
 It is straightforward
to verify that the right-hand-side of~\nn{trconn}  transforms according to~\nn{transf},
and this verifies the conformal invariance
of~$\nabla^\ct$. In the following we will usually write simply
$\nabla$ for the tractor connection. Since it is the only connection
we shall use on~$\ct$, its dual, and tensor powers, this should not
cause any confusion. Its curvature is  the tractor-endomorphism valued two form
$
{\mathcal R}^{\!\!\bm \sharp} \!
$;
in the above splitting this acts as
\begin{equation}\label{curvature}
{\mathcal R}_{ab}^{\!\!\bm \sharp}\left( \begin{array}{c}
\si\\[1mm]\mu_c\\[1mm] \rho
\end{array} \right) 
=
[\nabla_a,\nabla_b] 
\left( \begin{array}{c}
\si\\[1mm]\mu_c\\[1mm] \rho
\end{array} \right)  =
\left( \begin{array}{c}
   0 \\[1mm]
    W_{abc}{}^d \mu_d+ C_{abc} \sigma  \\[1.3mm]
    -C_{abc} \mu^c
    \end{array}\right)\, .
\end{equation}
Here $C_{abc}:=\nabla_a\Rho_{bc}-\nabla_b\Rho_{ac}$ denotes  the Cotton tensor.

 For~$[U^A]=(\si, \mu_a,
\rho)$ and~$[V^A]=(\tau, \nu_a,\kappa)$,  the conformally invariant {\em tractor metric}~$h$ on~$\mathcal{T}$  is given by
\begin{equation}\label{trmet}h(U,V)=h_{AB}U^A V^B=\si \kappa +\bg_{ab}\mu^a\nu^b+\rho\tau=:U\cdot V\,  .\end{equation}
Note that this has signature~$(d+1,1)$ and, as mentioned, is
preserved by the tractor connection, {\it i.e.},~$\nabla^\cT h=0$. 
It follows from this formula that~$X_A=h_{AB}X^B$ provides the surjection 
$X_A:\cT^A\to \ce[1]$.
   The tractor metric~$h_{AB}$
and its inverse~$h^{AB}$ are used to identify~$\ct$ with its dual in
the obvious way, equivalently, it is used to raise and lower tractor
indices. We will often employ the shorthand notation~$V^2$ for~$V_A V^A=h(V,V)$.

Tensor powers of the standard tractor bundle~$\ct$, and tensor parts
thereof, are vector bundles that are also termed tractor bundles. We
shall denote an arbitrary tractor bundle by~$\cT^\Phi$ and write
$\cT^\Phi[w]$ to mean~$\cT^\Phi\otimes \ce[w]$;~$w$ is then said to be
the weight of~$\cT^\Phi[w]$. 

Closely linked to~$\nabla^\ct$ is an important, second order,
conformally invariant differential operator $$ D^A : \Gamma (\cT^\Phi
[w])\to \Gamma (\cT^{\Phi'}[w-1])\, ,$$   known as the Thomas-D
(or tractor D-) operator. Here  $\cT^{\Phi'}[w-1]:=\cT^A\otimes \cT^\Phi[w-1]$. In a scale~$g$, 
\begin{equation}\label{Dform}
[D^A  ]_g =\left(\begin{array}{c} (d+2w-2)\, w\  \hspace{.3mm} \\[1mm]
(d+2 w-2) \nabla_a \\[1mm]
-(\Delta+ \J w)    \end{array} \right) ,
\end{equation}
where~$\Delta=\bg^{ab}\nabla_a\nabla_b$, and~$\nabla$ is the coupled
Levi-Civita-tractor connection~\cite{BEG,Thomas}. When $w=1-\frac{d}{2}$ we have
$D^A\stackrel g=-X^A (\Delta+[1-d/2]\J) $ where   $\Delta+[1-d/2]\J$ is conformally invariant; on densities this is the well-known Yamabe 
operator, so we term $w=1-\frac d2$ the {\it Yamabe weight}. 
The following variant of the Thomas D-operator is also useful.
\begin{definition}\label{hD}
Suppose that~$w\neq1-\frac d2$. The operator
$$
\hD^A:\Gamma(\cT^\Phi [w])\longrightarrow \Gamma\big(\ct^{\Phi'} [w-1]\big)
$$
is defined by
$$
\hD^A T :=  \frac1{d+2w-2} \, D^A T\, .
$$
\end{definition}

\begin{remark}\label{intYam}
Given a weight $w'$ tractor $V^A\in \Gamma(\ct^AM\otimes \ct^{\Phi'} M[w'])$ subject to
$$
X_A V^A=0\, ,
$$
where $[V^A]\stackrel {g} =(0,v_a,v)$  for some $g\in\cc$,
then we may extend the above definition to the projection along $V^A$ of the operator $\hD_A$   at the Yamabe weight $w=1-\frac{d}{2}$, by defining
$$
\begin{array}{cccc}
 V^A \hD_A :\!\!\!& \Gamma(\ct^\Phi M[1-\frac d2])&\longrightarrow& \Gamma(\ct^{\Phi'}M[w']\otimes \ct^{\Phi}  M[-\frac d2]) \\[1mm]
 &\rotatebox{90}{$\in$}&&\rotatebox{90}{$\in$}\\[1mm]
 &T&\stackrel g\longmapsto&\big(v^a \nabla_a +  [1-\frac d2]\, v\big)\, T\, .
 \end{array}
 $$
 {}[In~\cite{GoAdv} this invariant operator was denoted $V^A\widetilde D_A$.]
\end{remark}

Finally, the following Lemma is easily verified by direct application of Equation~\nn{Dform}:
\begin{lemma}
Let $T\in \Gamma(\ct^\Phi M[w])$. Then 
\begin{equation}\label{DXT}
D_A(X^AT)=(d+w)(d+2w+2)T\, .
\end{equation}
\end{lemma}

\subsection{Conformal hypersurfaces}\label{chyp}

We now consider
a hypersurface~$\Sigma$ smoothly embedded in a conformal manifold $(M,\cc)$, and term this a {\it conformal embedding of~$\Sigma$}.
Then the conformal structure $\cc$ on $M$ induces a conformal structure~$\bar \cc$ on $\Sigma$.
Moreover, working locally, we 
may assume that there is a section~$\hat n_a$ of $T^*M[1]$ such that ${\bg}_{ab}\hat n^a \hat n^b =1$ along~$\Sigma$; so $\hat n_a$ is the conformal analog of a Riemannian unit conormal field.

There is also a corresponding unit tractor
object (from~\cite{BEG}) called the {\em normal tractor}~$N^A$ which is defined along~$\Sigma$, in some choice of scale,  by
\begin{equation}\label{normaltractor}
[N^A]\stackrel\Sigma= \begin{pmatrix}
0\\\, \hat n_a\\ -H
\end{pmatrix}\, .
\end{equation}
This is the basis for  a conformal hypersurface calculus~\cite{Goal,Grant,Stafford,YuriThesis} which is further developed in the article~\cite[Section 4]{GW-willII}.
The most basic ingredient of this is 
the conformal tractor analog
of the Riemannian isomorphism between the intrinsic tangent bundle~$T\Sigma$ and the subbundle~$TM^\top\!$ of
~$TM|_\Sigma$ orthogonal to~$\hat n^a$ along~$\Sigma$; this relates
the
hypersurface tractor bundle~$\ct\Sigma$ and the ambient one~$\ct M$.
In fact (see~\cite{BrGoCNV} and~\cite[Section 4.1]{Goal}),  the subbundle~$N^\perp$ orthogonal to the normal tractor (with respect to the tractor metric~$h$)
along~$\Sigma$
is canonically isomorphic to the intrinsic hypersurface tractor bundle~$\ct\Sigma$.
In the same spirit as our treatment of the Riemannian case (see Section~\ref{standard}), we will use this isomorphism to identify these bundles and  use the same abstract index for~$\ct M$ and~$\ct \Sigma$. To employ this isomorphism for explicit computations in a given choice of scale,  we  the detailed
relationship between sections as given in corresponding splittings of the ambient and hypersurface tractor bundles is needed.
 In terms of sections expressed in a scale~$g\in \cc$ (determining $\bar g\in\bar\cc$), this isomorphism is given by the map
\begin{equation}\label{Tisomorphism}
\big[V^A\big]_g:=\begin{pmatrix}v^+\\[1mm]v_a\, \\[2mm]v^-\end{pmatrix}\mapsto
\begin{pmatrix}
v^+\\[1mm]v_a-\hat n_aH v^+\\[2mm]v^-+\frac12 H^2 v^+
\end{pmatrix}=\big[U^A{}_B\big]_{\bar g}^g\, \big[ V^B\big]_g=:\big[\bar V^A\big]_{\bar g}\, ,
\end{equation}
where~$V^A\in\Gamma (N^\perp)$ and~$\bar V^A\in  \Gamma(\ct \Sigma)$.
Here the~$SO(d+1,1)$-valued matrix
$$
\big[U^A{}_B\big]_{\bar g}^g:=\begin{pmatrix}1&0&\ 0\ \\[2mm]-\hat n_a H&\delta_a^b&0\\[3mm]-\frac12 H^2&\hat n^b H&1\end{pmatrix}\, ,
$$
and we have used the canonical isomorphism between~$T^*M\big|_\Sigma$ and~$T^*\Sigma$ defined by the (unit)  normal vector~$\hat n_a$
to identify sections of these. Note that for tractors along~$\Sigma$  in the joint kernel of~$X_\llcorner$ and~$N_{\llcorner}$ 
(contraction by the canonical and normal tractors), the map~\nn{Tisomorphism} is the identity.

\subsection{Defining densities}\label{geom-Sc}
The notion of a defining function adapts naturally to densities, as follows.

Given a hypersurface~$\Sigma$, a section~$\sigma\in\Gamma(\ce[1])$ 
is said to be a {\it defining density} for~$\Sigma$ if~$\Sigma=\Z(\sigma)$ and~$\nabla \sigma$ is nowhere vanishing along~$\Sigma$ where $\nabla$ is the Levi-Civita connection for some, equivalently  any,  $g\in\cc$.
For a defining density~$\sigma$, we define a corresponding tractor field
\begin{equation}\label{sctrac-def}
I^A_\si:=\hD^A \si\, .
\end{equation}
For later use we introduce some notation for the components of this in a scale:
$$ [\hD^A\si]\stackrel{g}{=} (\sigma,\nabla_a \si,-\frac{1}{d}(\Delta +\J)\si)=:(\sigma,n_a,\rho) .$$
Since $I^A_\si$ includes the full 1-jet of $\si$ it follows at once that 
it is nowhere vanishing.  In fact since we assume Riemannian signature
 for any defining density~$\sigma$ we have that
\begin{equation}\label{ge0}
I_\sigma^2>0
\end{equation}
holds in a neighbourhood of~$\Sigma$.

\subsection{Scale and almost Riemannian structure}\label{aR}
A choice of positive section~$\tau\in \Gamma(\ce_+[1])$ is
equivalent to a choice of metric from the conformal class~$\cc$ via
the relation~$g^\tau:=\tau^{-2}\bg$. Traditionally any such section is
thus termed a scale. However conformally compact manifolds, conformal
hypersurfaces, and related structures are naturally treated by working
with densities that may have a zero locus (and change sign). Thus we generalise our notion of scale as follows.

\begin{definition}\label{scale} On a conformal manifold $(M,\cc)$ any  section
 ~$\sigma\in\Gamma(\ce[1])$ such that $$I^A_\si:=\hD^A \si$$ is nowhere
  vanishing is called a {\em scale} and $I^A_\si$ is the corresponding
  {\em scale tractor}. We will describe such data $(M,\cc,\si)$
  (equivalently $(M,\cc,I^A_\si)$) as an {\em almost Riemannian
    structure} or almost Riemannian geometry.
 \end{definition}

On an almost Riemannian structure $\si= X_A I^A$ (denoting $I:=I_\sigma$) is non-vanishing on
an open dense set; so $\si$ determines a metric on such an open
set. On the other hand, because it is nonzero,~$I^A$ provides a
structure on the manifold~$M$ which connects the geometry of these
metrics to that of the zero locus of~$\si$. This idea was developed
in~\cite{Goal} (see also \cite{CurryG}).  Note that if $I^2$ has a
fixed strict sign then $I^A$ determines a structure group reduction of
the conformal Cartan geometry, {\it cf.}~\cite{CapGoH-duke} where holonomy
reductions are treated. 

Note that, in particular, if $\si$ is a defining density for a
hypersurface  then $\sigma$ is a scale away from its zero locus  $\Sigma$. Thus a conformally compact manifold $\,\overline{\!M}$ is
the same as an almost Riemannian geometry
$(\,\overline{\!M},\cc,\si)$, where $\si\in \Gamma(\ce[1])$ is a
defining density for the boundary $\partial M=\mathcal{Z}(\si)$.

On conformally compact manifolds one may consider the natural
curvature quantities on the bulk. In general these are not expected to
extend smoothly (or in some reasonable way) to the boundary since the
metric is singular there.  It is therefore important to determine the
extent to which such curvatures may have meaningful limiting values at
the conformal infinity. In fact such questions can be treated quite
mechanically by treating the structure as an almost Riemannian
geometry, and this is a main motivation for this notion.

This idea applies in particular to the scalar curvature, which lies at
the heart of our considerations here.
The following refers to the quantity $S(\si)$ in Equation~\nn{Ytwo}:
\begin{proposition}[see \cite{Goal}]\label{I2-prop}
For~$\si\in \Gamma(\ce[1])$ the quantity~$S(\si)$  is the squared length of the corresponding scale tractor: 
\begin{equation}\label{Isq}
S(\si)=I^2_\si:= h_{AB}I^A_\si I^B_\si.
\end{equation}
\end{proposition}

Thus the key Equation (\ref{Ytwo}) has a simplifying and geometrically useful
tractor fomulation
\begin{equation}\label{ASCond}I_\si^2=1\, . \end{equation}
This observation is critical for our later developments.  Meanwhile
observe that although the Equation~ (\ref{Ytwo}) is second order, it
is a natural conformal analogue of the Riemannian equation~$n_s^2=1$
of Problem~\ref{Riemannsfirststep}: in place of~$n_s=\dd s$ we have
$I_\si=\hD\si$. Moreover, if $I^2=1+\mathcal{O}(\si^2)$ then
$I_\si|_\Sigma=N$, the normal tractor of \nn{normaltractor}, see
\cite{Goal}.

If~$I_\si$ is any scale tractor for~$\si\in \Gamma(\ce[1])$, then away
from the zero set of~$\si$ we have~$I^2_\si=-
\Sc^{g^o}\!/\big(d(d-1)\big)$, where $g^o=\si^{-2}\bg$. Thus
Proposition \ref{I2-prop} shows that on smooth conformally compact
manifolds, the scalar curvature extends smoothly to the boundary.
Moreover, on~$M$, the condition~$I^2_\si={\rm constant}$ generalises the
condition of constant scalar curvature.  If this holds, we term the
manifold {\em almost scalar constant}~(ASC) and there are severe
restrictions on the possible zero loci~$\Z(\si)$ of~$\si$,
see~\cite{Goal}. In particular, if~$I^2_\si=1$ it is either the case
that $\Z(\si)$ is empty, or else a smoothly embedded hypersurface. If
not a boundary, the latter is separating. Thus the problem treated
here fits very nicely into this theory.

%% Moreover, for ASC structures with $\Z(\si)$ a hypersurface $\Sigma$,
%% we have that $n_a$ is a unit normal along $\Sigma $ and (see
%% Remark~\ref{namingmethod} below) the scale tractor~$I_\si$ agrees with
%% the conformally analogous normal tractor of
%% Expression~\nn{normaltractor}.

 Since our interests here concern conformally compact manifolds, and
 more generally conformal hypersurfaces (in the case of Riemannian
 signature) we will henceforth assume~\nn{ge0} holds on~$M$, without
 loss of generality.

\subsection{Laplace--Robin operator, $\frak{sl}(2)$  and tangential operators}\label{sl2structure}

Combining the scale tractor $I:=I_\sigma$ and the Thomas D-operator  gives 
the {\it Laplace--Robin} operator~\cite{powerslap,GoIP}
$$
I\cdot D:\Gamma(\ct^\Phi[w])\longrightarrow \Gamma(\ct^\Phi[w-1])\, .
$$
This is a canonical degenerate Laplace operator with degeneracy precisely along~$\Sigma=\Z(\sigma)$. Calculated in some scale $g\in\cc$
\begin{equation}\label{IdotD}
\begin{split}
I\cdot D&\stackrel g=\ 
(d+2w-2)(\nabla_n+w\rho)-\si (\Delta+w\J)\\[1mm] &=
-\sigma \Delta + (d+2w-2)\big[\nabla_n-\frac wd(\Delta \sigma)\big]-\frac{2w}{d}(d+w-1)\sigma\J\, ,
\end{split}
\end{equation}
where $n=\nabla\sigma$. Away from $\Sigma$, we may specialise this to the scale $g^o=\sigma^{-2}\bg$; trivialising density bundles accordingly gives a tractor-coupled Laplace-type operator
$$
I\cdot D\stackrel{g^o}=-\Big(\Delta-\frac{2w(d+w-1)}{d}\J\Big)\, .
$$
Conversely, along $\Sigma$, the Laplace--Robin operator becomes first order. In particular, when the scale tractor $I_\sigma$ obeys the ASC condition~\nn{ASCond} along~$\Sigma$,
\begin{equation}\label{RobinI}
I\cdot D=(d+2w-2)\delta_n\, ,
\end{equation}
where the first order operator 
$$
\delta_n :\stackrel g=\nabla_n-wH\, .
$$ 
 The above is precisely the conformally invariant Robin operator of~\cite{cherrier,BrGoCNV}.

In addition to unifying boundary dynamics and interior Laplace problems, the Laplace--Robin operator $I\cdot D$ and the scale~$\sigma$ are generators of a solution-generating $\frak{sl}(2)$ algebra~\cite{GW}. 
To display this, we first define a triplet of canonical operators.
\begin{definition}\label{xhy}
Let $\sigma\in\Gamma(\ce M[1])$ be a defining density
with 
nowhere vanishing $I^2$. Then we define the triplet of operators $x$, $h$, and $y_\sigma=:y$, mapping   $$\Gamma(\ct^\Phi M[w])\to \Gamma(\ct^\Phi M[w+\varepsilon])\, ,$$ where $\varepsilon=1,0$, and $-1$, respectively, and for 
 $f\in 
\Gamma(\ct^\Phi M[w])
$
$$
xf:=\sigma f\, ,\qquad h f:= (d+2w) f\, ,\qquad
y  f:=-\frac1{I_\sigma^2}\, I_\sigma\cdot D f\, .
$$
\end{definition}

\begin{proposition}[\cite{GW},  Proposition 3.4]\label{thesl2}
 The operators $\{x,h,y\}$ obey the ${\frak sl}(2)$ algebra
\begin{equation}\label{sl2}
[h,x]=2x\, ,\qquad [x,y]=h\, ,\qquad [h,y]=-2y\, .
\end{equation}
\end{proposition}

This operator  algebra was called a {\it solution generating algebra} 
in the setting of linear extension problems in~\cite{GW}, because it generates formal solutions via an  expansion in~$x$.
Here, it plays a similar {\it r\^ole} for the non-linear setting we treat.
In fact,
obstructions to smooth solutions of the extension problem $y f = 0$ (with appropriate conditions on $f$),
can be recovered from a 
 related tangential operator problem~\cite{GW} underpinned by the following, standard~$\frak{sl}(2)$  identities
\begin{equation}\label{yk}
[x,y^k]= y^{k-1}k(h-k+1)\, , \quad
[x^k,y]=x^{k-1} k(h+k-1)\, ,\quad
\mbox{ where } k\in {\mathbb Z}_{\geq1}\, .
\end{equation}
Tangential operators 
provide a link between ambient and hypersurface geometry:
\begin{definition}
Given a hypersurface~$\Sigma$ and a defining density~$\sigma$,   an operator~$P$, 
acting between smooth sections of vector bundles over~$M$, 
 is called {\it tangential}
if~$$P\circ x=x\circ \widetilde P\, ,$$ 
and~$\widetilde P$  is some smooth operator. 
\end{definition}
Observe that for smooth sections $f$ and $P$ tangential,  we have that $Pf|_\Sigma$ depends only on $f|_\Sigma$.
Thus tangential operators
canonically define hypersurface  operators, and hence our interest in them. 
The second identity in Equation~\nn{yk} implies that the operator $y^k$ is tangential acting on 
$\Gamma(\ct^\Phi M\big[\frac{k-d+1}{2}\big])$.
 This allows us to construct extrinsic conformal Laplacians along the hypersurface~$\Sigma$ in Section~\ref{EGJMS}, and in turn write holographic formul\ae\ for the obstruction density in Section~\ref{ASCObst}.

\subsection{Tractor calculus identities} \label{idents} 
We list here some identities that are required for the subsequent
discussion.  The key result is a characterization of how the Thomas D-operator
violates the Leibniz rule.

As a direct corollary of Lemma~\ref{leib0} from Appendix~\ref{FG} 
we have the following
useful rule for the Thomas D-operator acting on
products of tractors\footnote{This rule also appeared, in a rather different physics context, in~\cite{Joung}.}:
\begin{proposition}[Leibniz's failure]\label{leib-fail}
Let~$T_i\in \Gamma(\cT^\Phi M[w_{i}])$ for~$i=1,2$, and~$h_{i}:=d+2w_{i}$,~$h_{12}:=d+2w_1+2w_2-2$ with~$h_i\neq0\neq h_{12}$.
Then
\begin{equation}\label{ls}
\hD^A(T_1T_2)- (\hD^A T_1) \, T_2 - T_1(\hD^A T_2)=-\frac{2}{d+2w_1 + 2w_2 -2}\, X^A\, (\hD_B T_1)(\hD^B T_2)\, .
\end{equation}
\end{proposition}
From this we obtain the following: 
\begin{corollary}\label{leib-prod}
Let~$T_i\in \Gamma(\cT^\Phi M[w_{i}])$ for~$i=1,,\ldots,k$, and~$h_{i}:=d+2w_{i}$,~$h_{1\ldots k}:=d+2\sum_{i=1}^kw_i-2$ with~$h_i\neq0\neq h_{1\ldots k}$.
Then
\begin{equation}\label{lsmanytimesover}
\begin{split}
\hD^A(T_1T_2\ldots T_k)&-\sum_{i=1}^k T_1\ldots (\hD^A T_i) \ldots T_k\\&=
-\frac{2X^A}{d+2\sum_{i=1}^k w_i-2} \sum_{1\leq i<j\leq k} T_1 \ldots (\hD^B T_i) \ldots (\hD_B T_j) \ldots T_k\, .
\end{split}
\end{equation}
\end{corollary}
\begin{proof}
The result follows from a simple induction based on~\nn{ls}.
\end{proof}

Among the many identities that follow from the above proposition, two are key for our purposes:
\begin{lemma}\label{actonstuff}
Let~$T\in \Gamma(\cT^\Phi M[w])$ and~$k\in {\mathbb Z}_{\geq 0}$ with~$d+2k+2w-2\neq0\neq d+2w-2$. 
Then
\begin{equation} \label{hDsip}
\hD^A (\sigma^k T)-\sigma^k \hD^A T = k\,  \sigma^{k-1}I^A T
- \frac{2k\,  X^A\sigma^{k-1}I\cdot D \, T}{(d\!+2k\!+2w\!-2)(d+2w-2)}
-\frac{k(k-1)X^A\sigma^{k-2}I^2 T }{d+2k+2w-2} \, .
\end{equation}
\end{lemma}

\begin{proof}
This result is a direct application of Equation~\nn{lsmanytimesover}.
\end{proof}

An immediate corollary of this Lemma, for symmetrized tensor products, will be particularly useful:

\begin{corollary}
Let~$T\in \Gamma(\cT^\Phi M[w])$ and~$k\in {\mathbb Z}_{\geq 0}$ with~$d+2k+2w-2\neq0\neq d+2w-2$. 
Then
\begin{equation}\label{hD2}
\begin{split}
\Big(\hD_A (\sigma^k  T)\Big)\! &\odot\!\Big(\hD^A (\sigma^k  T)\Big) = \sigma^{2k}\big(\hD_A\,  T\big)\odot \big(\hD^A\,  T\big) \\[2mm]
&+  \frac{2k\,  (d-2)}{d+2k+2w-2}\, T\odot \Big[\frac{\sigma^{2k-1}\,  I\cdot D\, T}{d+2w-2} + \frac{(k\,  d+2w)\, \sigma^{2k-2}I^2 T}{2(d-2)} \Big]\, .
\end{split}
\end{equation}
\end{corollary}

Contracting Equation~\nn{hDsip} with the scale tractor~$I^A$ and stripping off the arbitrary tractor~$T$ gives the operator identity:
\begin{equation}\label{algebra}
[I\cdot D,\sigma^{k+1}] = I^2 \sigma^k\,(k+1)(d+2\w+k)\, .
\end{equation}
Here,~$\w=h-\frac d2$ is the weight operator~$\w:\Gamma(\cT^\Phi M[w])\to \Gamma(\cT^\Phi M[w])\ni T$ with~$\w T=wT$.

\begin{remark}\label{alphaC}
The above identity is also a direct consequence of the $\frak{sl}(2)$ algebra of Proposition~\ref{thesl2} ({\it cf}.  Equation~\nn{yk}).
The ambient space method used to establish Lemma~\ref{leib0} in Appendix~\ref{FG} can also be employed, for positive~$\sigma$, to extend the validity of Lemma~\ref{actonstuff}, its corollary  and Display~\nn{algebra} to~$k$ replaced by any~$\alpha\in {\mathbb R}$ (see~\cite[Equation 5.9]{GW}).
\end{remark}

\section{Conformal hypersurfaces and the extension problem}
\label{cyp}

Throughout the
following sections we shall assume that~$M$ is oriented, ~$\dim (M)\geq 3$, and that any
hypersurface~$\Sigma~$ is the zero locus of some smooth defining
function.

\subsection{Conformal hypersurface invariants} \label{chi}
In Section~\ref{definv-sec} we defined (scalar and tensor-valued)
invariants of a Riemannian hypersurface~$\Sigma$. We now require the
analogous concept for hypersurfaces in a conformal manifold.
\begin{definition}\label{chi-def}
A {\em weight~$w$ conformal covariant} of a hypersurface~$\Sigma$ is a
Riemannian hypersurface invariant~$P(\Sigma,g)$ with the property that
$P(\Sigma,\Omega^2 g)= \Omega^w P(\Sigma,g)$, for any smooth positive
function~$\Omega$.
Any such covariant determines an invariant
section of~$\ce \Sigma[w]$ that we
shall denote~$P(\Sigma; \bg)$, where~$\bg$ is the conformal
metric of the conformal manifold~$(M,[g])$. We shall say that~$P(\Sigma; \bg)$ is a {\em
  conformal invariant} of~$\Sigma$. When $\Sigma$ is understood
  by context, the term {\em hypersurface conformal invariant}
will refer to densities or weighted tensor fields which arise this way.

\end{definition}

\noindent In particular one may na\"{\i}vely attempt to construct
conformal hypersurface invariants by seeking linear combinations of
Riemannian hypersurface invariants of the form \nn{parcon} that have
the required conformal behaviour. But this method is intractible except at
the lowest order.

\subsection{The extension problem} \label{main}
We are now ready to treat the main problem. In the conformal setting
it is efficient to use defining densities rather than defining
functions to describe hypersurfaces analytically.

If~$\sigma$ is a defining density for some hypersurface~$\Sigma$, and a background metric ($g \in \cc$
on~$M$) is chosen, we shall write~$n_\si:=\nabla \si$ (or simply~$n$
if~$\si$ is understood), where~$\nabla$ is the Levi-Civita connection. 
\begin{definition}\label{normal-den}
A defining density~$\si$ for~$\Sigma$ is said to  be {\em normal}\, if~$$n_\sigma^{\, 2}:=\bg^{-1}(n_\sigma,n_\sigma)=1+\O(\sigma)\, .$$
\end{definition}
\noindent Note that if~$s$ is a normal defining function in the scale 
$g=\tau^{-2}\bg$, where~$\tau\in \Gamma(\ce[1])$ is nowhere zero, then~$\si:=s\tau$ is a
normal defining density for~$\Sigma$. In particular
$\Sigma$ has a normal defining density.

Observe now that if~$\si$ is a defining density for~$\Sigma$ then,
computing in terms of any background metric, we have~$I^2_\si = n^2
+2\si \rho$ for some density~$\rho\in \Gamma(\ce [-1])$. This follows
from~\nn{trmet} and~\nn{sctrac-def}. So~$\si$ is a normal defining
density if and only if
$$
I^2_\si =1 + \O(\sigma),
$$
and this statement is conformally invariant.

Thus we may always find~$\si$ such that~$I^2_\si =1 + \si A_1$ for some
smooth density~$A_{1}\in \Gamma(\ce[-1])$.  Now, by Proposition~\ref{I2-prop},
Problem~\ref{conffirststep} is equivalent to the following.
\begin{problem}\label{I2-prob} 
Find a smooth defining density~$\bar \sigma$ such that
\begin{equation}\label{ind}
I^2_{\bar\si}=1 + \bar{\sigma}^{\ell} A_\ell\, ,
\end{equation}
for some smooth~$A_\ell\in \Gamma(\ce[-\ell])$,
where~$\ell \in\mathbb{N}\cup\infty$ is as high as possible.
\end{problem}
To treat this it is natural to  set up a recursive approach similar to that used to solve Problem~\nn{Riemannsfirststep}.
\begin{lemma}\label{Isquare}
Suppose~$\sigma\in \Gamma(\ce[1])$  defines the  hypersurface~$\Sigma$ in~$(M,\cc)$,  and 
~$$
 I_\sigma^{\, 2}=1+\sigma^k A_k\, , \qquad A_k\in \Gamma(\ce[-k])\, , \quad k\geq 1.
~$$
Then, if~$k\neq d$,  there exists~$f_k\in \Gamma(\ce[-k])$, unique to ${\mathcal O}(\sigma)$, such that the scale tractor~$I_{\sigma'}$ of the new defining density~$\sigma':=\sigma+\sigma^{k+1} f_k$ obeys
$$
I_{\sigma'}^{\, 2}=1+\O(\sigma^{k+1})\, .
$$

If ~$k=d$, then for any~$f\in\Gamma(\ce[-d])$ and~$\sigma':=\sigma+\sigma^{d+1} f$, we have
$$
I_{\sigma'}^{\, 2}=I_\sigma^{\, 2}+\O(\sigma^{d+1})\, . 
$$
\end{lemma}

\begin{proof}
Let~$I:=I_\sigma$ and first observe 
$$
\big(\hD \sigma'\big)^2=I^2 +\frac2d \, I\cdot D \big(\sigma^{k+1} f_k\big) + \Big[\hD \big(\sigma^{k+1} f_k\big)\Big]^2\, .
$$
Then using Eq.~\nn{hD2} and~\nn{algebra} we have
\begin{equation}\label{induct}
\big(\hD \sigma'\big)^2=1+\sigma^k A_k + \frac2d\,  \sigma^k\, (k+1) (d-k) f_k +\O(\sigma^{k+1})\, . 
\end{equation}
Thus, when~$k\neq d$, setting \begin{equation}\label{A2f} f_k=-\frac{d}{2(d-k)(k+1)}\,  A_k\, ,\end{equation} gives the first result. The result at~$k=d$ follows directly from Eq.~\nn{induct}.
\end{proof}

The first main result now follows:
\begin{theorem}\label{obstr}
There is a distinguished defining density~$\bar\sigma\in \Gamma(\ce[1])$, unique to~$\O(\sigma^{d+1})$ where~$\sigma$ is any given defining density, such that 
\begin{equation}\label{ddens}
I^{\, 2}_{\bar \sigma}=1+\bar\sigma^d B_{\bar \sigma}\, , 
\end{equation}
for some smooth density~$B_{\bar \sigma}\in \Gamma(\ce[-d])$.
Moreover~$\bar \sigma$ may be given by a canonical formula~$\bar\sigma(\sigma)$ depending (smoothly) only on~$(M,\cc,\sigma)$.
\end{theorem}

\begin{proof}
Existence follows from the Lemma, by induction, while uniqueness
follows from Eq.~\nn{A2f} determining the improvement term of the form
$\sigma^{k+1} f_k$ where~$f_k=c_k A_k+\O(\sigma)$. The constant~$c_k$ is
determined and non-zero (when~$k\neq d$) at each order. The last
statement is immediate by the construction given in the induction.
\end{proof}

Given the above theorem, we make the following definition:

\begin{definition}\label{cunit}
We say that a smooth defining density~$\bar\sigma\in \Gamma(\ce[1])$, is a 
{\it conformal unit  defining density}  if it obeys
\begin{equation}\label{I2def}I^{\, 2}_{\bar \sigma}=1+\O(\bar\sigma^d)\, .\end{equation}
\end{definition}

\begin{remark}
It follows from Theorem~\ref{obstr}
that a conformal hypersurface embedding determines an
ambient metric (singular along the hypersurface) from the conformal class, up to the order given.
This has a host 
of applications. For example, 
as we shall later show, 
given a conformal unit defining density, there then exist distinguished extensions of hypersurface quantities.
%
%The existence of conformal unit defining densities implies that
%extrinsic hypersurface quantities can be continued off~$\Sigma$, at
%least to accuracy commensurate with that to which the ASC
%condition~$I_{\bar \sigma}^2=1$ is solved. This implies a slew of
%identities for normal derivatives~$\nabla_{\bar n}$ (where~$\bar
%n:=\nabla \bar \sigma$) along $\Sigma$ of hypersurface quantities.
%; examples of these are given in Section~\ref{comp-ext}.
\end{remark}

\begin{theorem}\label{main-Th}
Given two normal defining densities~$\sigma$ and~$\sigma'$, and~$B_{\bar \sigma}$ and~$\bar\sigma(\cdot )$  as determined in Theorem~\ref{obstr} above, then
$$
B_{\bar \sigma(\sigma)}=B_{\bar\sigma(\sigma')}+\O(\sigma)\, .
$$ 
\renewcommand{\B}{{\mathcal B}}
Thus~$\B:=B_{\bar\sigma( \cdot )}\big|_\Sigma$ is independent of~$\sigma$ and determined by~$(M,\cc,\Sigma)$.
\end{theorem}

\begin{proof}
If~$\sigma$ and~$\sigma'$ are two defining densities then, by the uniqueness statement for the~$f_k$ in the first part of Lemma~\ref{Isquare}, given a defining density~$\sigma$   we have
$$
\bar\sigma(\sigma')-\bar\sigma(\sigma)=\bar \sigma^{d+1} f\, ,$$ where~$f\in \Gamma(\ce[-d])$.
The result now follows from the last statement of Lemma~\ref{Isquare}.
\end{proof}

We shall call the hypersurface conformal invariant~$\B$ of Theorem
\ref{main-Th} the {\em ASC obstruction density}.

If~${\bar \sigma_{k-1}(\sigma)}$ obeys~$I_{\bar \sigma_{k-1}}^{\,
  2}=1+\O(\bar\sigma_{k-1}^{k})$, then~$$\frac{I_{\bar
    \sigma_{k-1}}^{\, 2}-1}{\bar \sigma_{k-1}^{\, k}}$$ is smooth
along~$\Sigma$. Thus, using Equations~\nn{induct} and~\nn{A2f} we set 
$$
\bar\sigma_{k}\, :=\ \bar\sigma_{k-1}\, \left[\, 1-\frac{d\, }{2}\frac{I_{\bar \sigma_{k-1}}^{\, 2}-1}{(d-k)(k+1)}\, \right]\, , \quad\quad k\neq d.
$$ 
Iterating this gives an explicit formula for the order~$\sigma^d$
solution to the boundary 
Problem~\ref{I2-prob}, as well
as its obstruction, which we summarise below.
\begin{proposition}\label{prod}
Let~$\sigma=:\bar\sigma_0$ be a normal defining density
for~$\Sigma$. Then Problem~\ref{I2-prob} is solved to order
$\ell=d$ by 
recursive use of the formula
$$\bar\sigma_k:=\sigma\,  \prod_{i=0}^{k-1}\left[\, 1-\frac{d}{2}\frac{I_{\bar \sigma_i}^{\, 2}-1}{(d-i-1)(i+2)}\, \right]\, ,$$
with also 
$$
I^2_{\bar\si_i}= (\hD^A \bar\si_i) \hD_A \bar\si_i , \quad \quad 1\leq k<d-1.
$$
The ASC obstruction density is the obstruction to solving Problem~\ref{I2-prob}  smoothly, and is given by 
\begin{equation}\label{obstruction}
\B=\left.\left[\frac{I_{\bar\sigma_{d-1}}^{\, 2}-1}{\sigma^{d}}\right]\right|_\Sigma\, .
\end{equation}
\end{proposition}

\newcommand{\Ot}{\tilde\O}

\subsection{The higher order expansion and log terms}

To continue
the solution beyond order~$\ell=d$, we relax our smoothness
requirement to allow log terms. In particular we consider~$\log
\sigma$ which is defined as a section of a log density bundle; see
Section~2.1 of~\cite{GW}. Such a section is well defined away from~$\Sigma$ and
obeys the following generalization of the algebra~\nn{algebra}
\begin{equation}\label{logalgebra}
[I\cdot D,\log\sigma]=\frac{\, I^2}\sigma \, (d+2\w-1)\, .
\end{equation}
Thus we refine the notion of an order~$\O(\sigma^\ell)$ solution to allow for any finite power of~$\log\sigma$. 
Since we are solving for a density~$\sigma$, following~\cite{GW}, we must also introduce a second, {\it true scale}~$\tau\in \Gamma(\ce[1])$, {\it {\it i.e.}} a scale~$\tau$ that is nowhere vanishing on~$M$.
Thus, while~$\log \tau$ is a log density,~$\log(\sigma/\tau)=\log\sigma-\log\tau$ is a section of~$\ce M[0]$.

\begin{definition}
If densities~$f$ and~$g$ obey
$$
f=g+\sigma^\ell \sum_{j=0}^k\big(\log(\sigma/\tau)\big)^j C_j\, ,
$$
where~$C_j$ are smooth densities,~$k$ is any non-negative integer and $\tau$ is a true scale we write
$$
f=g+\Ot(\sigma^\ell)\, .
$$
\end{definition}

\medskip

\begin{proposition}\label{log}
Let~$\sigma$ be a conformal unit defining density for~$\Sigma$.
Then 
$$
\sigma':=\sigma \, \left[1+\frac d2\  {\log(\sigma/\tau)}\ \frac{I^2-1}{d+1}\right]\, ,
$$
obeys
$$
I_{\sigma'}^{\, 2}=1+\sigma^{d+1}\big[\log(\sigma/\tau)\, A + B'\big]=1+\Ot(\sigma^{d+1})\, ,
$$
for some smooth~$A,B'\in\Gamma(\ce[-d-1])$.
\end{proposition}

\begin{proof}
Because $\sigma$ is a conformal unit density we have $I_\sigma^2=1+\sigma^d B$.
The proof  now mimics  that of Lemma~\ref{Isquare}, save that when calculating the analog of Equation~\nn{induct}, one needs to apply 
the algebra~\nn{algebra} in conjunction with~\nn{logalgebra} to compute
$$
I\cdot D \,  \sigma^{d+1} \log(\sigma/\tau) B =
\sigma^{d+1} I\cdot D (\log\sigma-\log\tau) B =
- \sigma^d (d+1) I^2 B + \Ot(\sigma^{d+1})\, .
$$
Here the first equality relies on  the fact that~$\log(\sigma/\tau) B\in\Gamma(\ce[-d])$.
\end{proof}
\begin{remark}
Note that the obstruction density $B$ appears as the coefficient of the first logarithm term in a solution~$\bar\sigma$ to 
Problem~\ref{I2-prob} modified to require $I^2_\sigma=1+\tilde{\mathcal O}(\sigma^{d+1})$, namely
\begin{equation}\label{log-coeff}
\bar \sigma= \sigma + \sigma^{d+1}\Big[f+\frac{d}{2(d+1)}\log(\sigma/\tau) B\Big]+\tilde{\mathcal O}(\sigma^{d+2})\, ,
\end{equation}
where $\sigma$ is a conformal unit density and $f$ is a smooth density.
When the log term is present, by choosing the true scale $\tau$ appropriately, we may arrange for the density $f$ to vanish.
\end{remark}

An inhomogeneous version of the recursion of Proposition~\ref{prod} now  solves   Problem~\ref{I2-prob}, weakened to include log terms but now to all orders beyond the obstruction. The failure for the solution so obtained to  be unique is parameterised by the choice of true scale $\tau$.
\begin{proposition}
Let $k\geq d+1$ and
suppose the density $\sigma$ satisfies
$$
I^2_\sigma=1+\sigma^{k}\big[\log(\sigma/\tau)\, A + B^\prime+\tilde{\mathcal O}(\sigma)\Big]\, ,
$$
for smooth $A,B^\prime \in\Gamma(\ce[-k])$.
Then
$$
\sigma^\prime = \sigma \, \left[\, 1-\frac{d\, }{2}\frac{I_{\sigma}^{\, 2}-1}{(d-k)(k+1)}\, \right] + \frac d2 \frac{(d-2k-1)A}{(d-k)^2(k+1)^2}\, \sigma^{k+1} 
$$
obeys
$$
I^2_{\sigma^\prime}=1+\sigma^{k+1}
\big[\log(\sigma/\tau)\, A^\prime + B^{\prime\prime}+\tilde{\mathcal O}(\sigma)\Big]\, ,
$$
for smooth $A^\prime,B^{\prime\prime}\in\Gamma(\ce[-k-1])$.
\end{proposition}

\begin{proof}
The proof again closely mimics that of Lemma~\ref{Isquare}. We begin by computing the square of the scale tractor for an ansatz $\sigma^\prime=\sigma+\sigma^{k+1}F$ where $F=\big[f+\log(\sigma/\tau)\, g\big]$, and find (calling $I:=I_\sigma$)
$$
I_{\sigma^\prime}^2=\big(\hD \sigma^\prime\big)^2=1+\sigma^{k}\big[\log(\sigma/\tau)\, A + B^\prime+\tilde{\mathcal O}(\sigma)\Big]+
\frac2d\, I\cdot D\Big(\sigma^{k+1} F\Big)
+\tilde{\mathcal O}(\sigma^{k+2})\, .
$$
Now, Equation~\nn{algebra} implies
the operator identity
$$
I\cdot D\circ \sigma^{k+1}=\sigma^{k}\circ \big[(d-k)(k+1)+\sigma  I\cdot D\big]+\tilde{\mathcal O}(\sigma^{d+k})\, ,
$$
and in turn Equation~\nn{logalgebra} yields
$$
\big[(d-k)(k+1)+\sigma  I\cdot D\big]F=
(d-k)(k+1)F+(d-2k-1)g + \tilde{\mathcal O}(\sigma) \, .
$$
Noting that $\log(\sigma/\tau)\, A + B^\prime=\frac{I^2-1}{\sigma^k}+\tilde{\mathcal O}(\sigma)$, altogether we have
$$
I_{\sigma^\prime}^2-1= I^2-1+\frac 2d\, \sigma^k\,  \Big[ (d-k)(k+1)  F + (d-2k-1) g\Big]+\tilde{\mathcal O}(\sigma^{k+1})\, .
$$
Cancelling  the leading $\sigma^k\log(\sigma/\tau)$ terms on the right-hand-side  above requires $g=-\frac{d}{2}\frac{A}{(d-k)(k+1)}$ which gives the solution for $F$
$$
\sigma^k F=-\frac{d}{2}\frac{I^2-1}{(d-k)(k+1)} + \frac d2 \frac{(d-2k-1)A}{(d-k)^2(k+1)^2}\, \sigma^{k}\, . 
$$
Inserting the above in  $\sigma^\prime=\sigma(1+\sigma^k F)$ gives the quoted result.
\end{proof}

\begin{remark}
If the obstruction~$\B$ is absent, so that~$I^{\, 2}-1=\O(\sigma^{d+1})$, 
then the log term of the solution~\nn{log-coeff} can be omitted
and Proposition~\ref{prod} can be extended to give a {\it smooth}   order $\ell=\infty$ 
solution to Problem~\ref{I2-prob}.
In that case, there remains the freedom to modify the solution by the
term~$\sigma^{d+1} f$ in Equation~\nn{log-coeff}, for any smooth weight~$-d$ density~$f$.
\end{remark}

\subsection{Examples}\label{examples}
It is not difficult to compute explicit formul\ae\ for the obstruction
density by using the recursion defined in Proposition~\ref{prod} and
the tools of Section \ref{R-viaunit}. However with increasing
dimension (and hence order) the computations and expressions rapidly
become complicated, see~\cite{GW-willII,GGHW} for details. We give two examples. For simplicity the second is
given in conformally flat ambient spaces, in terms of a flat ambient
metric.
\begin{proposition}\label{Willmore}
For surfaces in dimension~$d=3$, the ASC obstruction density is given by
\begin{equation}\label{ASCobst2}
{\mathcal B}=-\frac{1}{3}\Big(\nablab_a\nablab_b +H\IIo_{ab}+\Rho_{ab}^\top\Big)\IIo^{ab}\, .
\end{equation}
%% For surfaces in   conformally flat three-manifolds, 
%% the ASC obstruction density is expressed, using a flat scale, by
%% $$
%% {\mathcal B}=-\frac13\big(\bar\Delta H + H\IIo_{ab}\IIo^{ab}\big)\, .
%% $$

For hypersurfaces in conformally flat four-manifolds,  
the ASC obstruction density is expressed, using a flat scale, by
$$
{\mathcal B}=\frac1{6}\Big((\nablab_c\IIo_{ab})^2+
2\IIo^{ab}\bar\Delta\IIo_{ab}+\frac32\, (\nablab^b\IIo_{ab} )\nablab^c\IIo^a{}_c
-2\bar\J\, \IIo_{ab}\IIo^{ab}  
+(\IIo_{ab}\IIo^{ab})^2
\Big)\, .
$$
\end{proposition}

\section{Related linear problems and the form of the obstruction density}
\label{lin}

By construction the obstruction density~$\cB$ found in Theorem~\ref{obstr} above
 is conformally invariant. We will establish in Theorem \ref{Bnatthm} below, that it is moreover a hypersurface 
conformal invariant in the sense of Definition \ref{chi-def}. Here we establish its
structure at  leading  differential order. 
\begin{theorem}\label{leading} Up to a non-zero constant multiple,~the ASC obstruction density $\cB$ 
takes the form:
$$ 
\bar{\Delta}^{\frac{d-1}{2}} H + \mbox{\rm  lower order terms}, \quad \mbox{if } d-1 \mbox{ is even} ;
$$ 
and 
$$ 
\mbox{\rm a fully non-linear expression}, \quad \mbox{if } d-1 \mbox{ is odd} . 
$$ 
\end{theorem}
\noindent When $d-1$ is odd,   this means that there is an
expression for~$\cB$ as a linear combination of terms, none of which is
linear in the jets of the ambient curvature~$R$ and the conormal~$n$.

In Lemma~\ref{rho-engine} 
 and Lemma~\ref{main-calc} below, we give a general algorithm 
for computing a formula for the obstruction density.
To calculate its  leading term 
we linearise  by computing its
infinitesimal variation. It is easily seen, using the algorithm there,
that every term in the expression will involve jets of the conormal
$n$.  Thus  it here suffices to consider an
$\mathbb{R}$-parametrised family of embeddings of~$\mathbb{R}^{d-1}$
in~$\mathbb{E}^{d}$, with corresponding defining densities~$\si_t$
such that the zero locus~$\Z(\sigma_0)$ is the~$x^{d}=0$ hyperplane
(where~$x^a$ are the standard coordinates on
$\mathbb{E}^{d}=\mathbb{R}^{d}$) and so that~$\cB|_{t=0}=0$. Then applying
$\delta:=\frac{d}{d t}\big(\, \cdot\, \big)\!\mid _{t=0}$ (often denoted by a dot) we obtain
the following:

\begin{proposition}\label{Bnature}
The variation of the obstruction density is given by
$$
\dot \cB = \left\{
\begin{array}{ll}
 a\, \bar{\Delta}^{\frac{d+1}2} \dot{\si} + \mbox{\em lower order terms}\, ,&d-1 \mbox{ even, with~$a\neq 0$ a constant, }\\[2mm]
\mbox{\em non-linear terms}\, ,& d-1 \mbox{ odd.}
\end{array}\right.
$$
\end{proposition}
This establishes  Theorem~\ref{leading} because---remembering that ${\cB}$ depends polynomially on~$\sigma$ and its derivatives---the
highest order term in the variation of mean curvature is~$\frac{1}{d-1}\bar{\Delta} \dot \si$. It also shows that when~$d-1$ is
odd the formula for~$\cB$, determined by Theorem~\ref{Bnatthm}  
below, has no
linear term.

\begin{proof}[Proof of Proposition~\ref{Bnature}]
First,  via Theorem
\ref{obstr}, for each~$t\in \mathbb{R}$ we can replace~$\si_t$ with the
corresponding normalised defining density~$\bar{\si}_t$ which solves
\begin{equation}\label{tsol}
I^2_{\bar{\si}_t}=1+{{\bar{\si}}_t}^{d}\, B_{\bar{\si}_t}\, .
\end{equation}
 We can
assume the family~$\bar{\si}_t$ depends smoothly on~$t$ since, according to Proposition~\ref{prod}, we may take
$\bar{\si}_t$ to depend polynomially on~$\si_t$.

Next observe that~$B_{\bar{\si}_0}|_{\Z(\bar\si_{t=0})}=0$. This
follows from Theorem~\ref{obstr} 
since, in this Euclidean hyperplane setting, there is a parallel standard tractor~$I$
such that~$I^2=1$, and with~$\si:=X^AI_A$ also a defining
density for~$\Z(\bar\si_{t=0})$.
(This follows, using stereographic projection, from 
the results in~\cite[Section~5]{Goal}.) In fact given a hypersurface defined by some defining
density~$\si$, the freedom to change~$\si$ is just that of multiplying
by a nowhere vanishing function. Thus there is no loss of generality
in assuming that our initial parametrised family~$\si_t$ obeys
$\si_{t=0}=\si$, and we shall henceforth take this to be the
case. Then, according to Proposition~\ref{prod},~$\bar{\si}_{t=0}=\si$.

Now we consider the variation at~$t=0$, through the family of
embeddings~$\bar{\si}_t$. Viewing the space~$\mathbb{E}^{d}$ as the~$t=0$
hypersurface in~$\mathbb{E}^{d}\times \mathbb{R}$ and  applying
$\delta$
to Expression~\ref{tsol} we have
\begin{equation}\label{tder}
 2 I\cdot\hD \, \dot{\si}= {\si}^{d}\, \dot B_{\si}\, ,
\end{equation}
where now~$I:=I_\si$ is parallel,~$\dot{\si}=\delta \bar{\si}$, and
$\dot{B}_{{\si}}=\delta B_{\bar{\si}_t}$. Given that (for each
$t$) the solution~$\bar{\si}_t$ is polynomial in the jets of the
original~$\si_t$, it is clear that the linearisation of~\nn{tsol}
provides a solution of the linear problem~$ I\cdot\hD \dot{\si} =0$, to
the given order, and the linearisation~$\dot{B}_{\si}$ of~$B$ is an
obstruction to the linear problem.  But the linear problem~$ I\cdot\hD
\dot{\si} =0$, for extending conformal weight 1-densities off a
hypersurface, is treated in~\cite{GW}. From
\cite[Proposition 5.4]{GW} this implies that when~$d-1$ is odd
\begin{equation}\label{B=Pndotsigma}
\dot{B}_{\si}\!\mid_{\Z(\si)} \, = 0, 
\end{equation}
as the linear problem is unobstructed in this case, while when~$d-1$ is even
$$ \dot{B}_{\si}\!\mid_{\Z(\si)}\,  = a\, \Pop_{d+1} \dot\si , \quad
\mbox{with~$a\neq 0$ a constant}\, ,
$$ 
where~${P}_{d+1}$ is the order~$d+1$, 
conformally invariant, Laplacian power operator~\cite{JV,EastwoodRice,GJMS} along~$\Z(\sib_0)$. In the flat Euclidean metric on~$\Z(\si)$, we have simply~$\Pop_{d+1}= \bar\Delta^{\frac{d+1}{2}}$. 
So we are done.
\end{proof}

\begin{remark}
Note that conformal Laplacian operator~$\Pop_{d+1}$  is beyond the
order for which such operators exist in general curved backgrounds
\cite{GJMS,Grnon,GoHirachi}. But above the proof is in the
conformally flat setting where the existence of $\Pop_{d+1}$ is
well known. In fact, for an even dimensional hypersurface (with
nowhere null conormal) in any conformal manifold the tangential
operators~$(I\cdot D )^\ell$ (on suitably weighted densities or
tractor fields) give conformal Laplacian type operators on~$\Sigma$,
for all even orders~$\ell$. This does not  contradict
the non-existence of higher order conformal Laplacian operators,  as these
use the data of the conformal embedding.
\end{remark}

 \subsection{Bernstein--Gel'fand--Gel'fand complexes}
\label{BGG}

\newcommand{\tstar}{\mbox{\Large $ \star$}}
\newcommand{\mstar}{\mbox{\large $ \star$}}

It is useful to see how the linearised operator of Proposition
\ref{Bnature} fits into the standard theory of conformally invariant
differential operators on conformally flat manifolds. 

We work on a manifold of dimension $n$ (which one should view as $d-1=\dim(\Sigma)$
for comparison with the above).  For this section, it will be
convenient to introduce an alternative notation for the
bundles and corresponding smooth section spaces  of certain tensor
bundles. We will use $\ce^k$ as a convenient alternative notation for
$\Omega^kM=\Gamma(\wedge^k T^*M)$. The tensor product of $\ce^k\otimes \ce^\ell$, $
\ell\leq n/2$, $k\leq \lceil n/2 \rceil$, decomposes (as bundles) into
irreducibles. We denote the highest weight component by
$\ce^{k,\ell}$.  (Here ``weight'' does not refer to conformal weight,
but rather the weight of the inducing O$(n)$-representation.) We
realise the tensors in $\ce^{k,\ell}$ as trace-free covariant
$(k+\ell)$-tensors $T_{a_1\cdots a_kb_1\cdots b_\ell}$ which are skew
on the indices $a_1\cdots a_k$ and also on the set $b_1\cdots
b_\ell$. Skewing over more than~$k$ indices annihilates~$T$, as does
symmetrising over any 3 indices.  Then as usual, we write~$\ce^{k,\ell}[w]$ as a shorthand for the tensor product
$\ce^{k,\ell}\otimes \ce[w]$. For example, $\ce^{2,2}[2]$ is the bundle
of {\em algebraic Weyl tensors} while the trace-free second
fundamental form $\IIo$ is a section of~$\ce^{1,1}[1]$.  Then we use the notation
$\ce_{k,\ell}[w]$ to stand for $\ce^{k,\ell}[w+2k+2\ell-n]$. This
notation is suggested by the global duality between $\ce^{k,\ell}[w]$ and
$\ce_{k,\ell}[-w]$ given by contraction and integration over the
manifold (in the case of compact manifolds, or else by using compactly
supported sections).

The round sphere is conformally flat,  and the sphere  equipped with the
conformal class of the round metric provides the basic ``flat model''
for conformal geometry in each dimension. This is acted upon by the
conformal group $SO(n+1,1)$;  the linear differential
operators (between irreducible bundles) that intertwine this action are
classified via corresponding dual Verma module homorphisms
\cite{BC,EastSlov}.  Most of the existing operators arise as  differentials in complexes known as BGG complexes \cite{Lepowsky}. The BGG
complex of interest to us   takes (for dimensions $n\geq 8$) the form
$$
 \begin{picture}(350,50)(35,-32)
 \put(10,0)
 {$
\ce[1]\xrightarrow{\sf aE}
\ce^{1,1}[1] \xrightarrow{\sf Cod}\ce^{2,1}[1] \to 
\cdots \to \ce_{2,1}[-n-1] \xrightarrow{{\sf Cod}^*} \ce_{1,1}[-n-1] 
\xrightarrow{{\sf aE}^*}\ce[-n-1] ~.
 $}                                                         %the seq
 \end{picture} 
$$
\vspace{-14mm}

\noindent
 The
details of most of the operators will not be important for us.  For
the named operators: $\operatorname{\sf Cod}$ is the conformal Codazzi
operator and is given in a scale in Equation~\nn{CodazziOperator} while $\operatorname{\sf
  aE}:\ce[1]\to \ce^{1,1}[1]$ is the so-called almost Einstein
operator, which (in a scale) is given by $\si \mapsto
\nabla_{(a}\nabla_{b)\circ}\, \si +P_{(ab)\circ}\, \si$, while $\operatorname{\sf
  Cod}^*$ and $\operatorname{\sf aE}^*$ are formal adjoints of these.

In odd dimensions the differentials of the BGG complex exhaust the list of all
conformally invariant differential operators between the bundles
concerned. However in even dimensions there are also {\em long
  operators} ${\sf L}:\ce^{k,\ell}[1]\to \ce_{k,\ell}[-1]$, and an additional
pair of operators about the centre of the diagram~\cite{BC}.  
Altogether this gives the operator diagram
$$
\begin{picture}(350,50)(29,-32)
\put(10,0)
{$
\ce[1]\xrightarrow{\sf aE}
\ce^{1,1}[1]\xrightarrow{\sf Cod}\ce^{2,1}[1] \to 
\cdots \to \ce_{2,1}[-n-1] \xrightarrow{{\sf Cod}^*}\ce_{1,1}[-n-1] 
\xrightarrow{{\sf aE}^*} \ce[-n-1] 
$}                                                         %the seq
\put(18,-31){\line(1,0){354}}
\put(18,-31){\line(0,1){24}}                               %long arrow
\put(372,-31){\vector(0,1){24}}
\put(32,-27){\scriptsize${\sf P}_{n+2}$}

\put(65,-23){\line(1,0){230}}
\put(65,-23){\line(0,1){16}}                               %one out arrow
\put(295,-23){\vector(0,1){16}}
\put(73,-19){\scriptsize$\sf L$}

\put(123,-14){\line(1,0){86}}
\put(123,-14){\line(0,1){7}}                               %middle arrow 
\put(209,-14){\vector(0,1){7}}
\put(1453,-10){\scriptsize${}$}

\end{picture}
$$
for dimensions 8 or greater.  The operators in this diagram are unique
(up to multiplying by a constant), and the diagram indicates by arrows
all the operators between the bundles explicitly presented.  In particular, all compositions  shown vanish.  The same diagram applies in
dimensions 6 and 4 with minor adjustments: In dimension 6 there are
two ``short'' operators with domain $\ce^{2,1}[1]$ and two with range
$\ce_{2,1}[-1]$. {}Built from these there is one non-trivial composition $
\ce^{2,1}[1]\to \ce_{2,1}[-1]$.  Similiarly in dimension 4 we have
$\tstar \circ {\sf Cod}: \ce^{1,1}[1] \to\ce^{2,1}[1]$ and ${\sf
  Cod}^* \circ \tstar : \ce^{2,1}[1]\to \ce_{1,1}[-1] $, as well as
the operators indicated. Here $\star$ is a bundle involution related
to the Hodge star operator on middle degree forms.  Then ${\sf L}$ is
the composition ${\sf Cod}^*\circ {\sf Cod}$.  In dimension 2 the
corresponding diagram is
$$
\begin{picture}(350,50)(-78,-28)
\put(5,0){$
\ce[1]  
\begin{picture}(20,20)(0,0)
\put(0,3){$\xrightarrow{\hspace{1mm} \sf aE \hspace{1.1mm} }$}
\put(0,-3){$\xrightarrow[\mstar {\sf aE}]{}$}
\end{picture}\, \  
\ce^{1,1}[1] \!
\begin{picture}(20,20)(-5,0)
\put(0,3){$\xrightarrow{{\sf aE}^*\mstar}$}
\put(0,-3){$\xrightarrow[\hspace{1mm}{\sf aE}^*\hspace{1mm}]{}$}
\end{picture}
\ \ \ \, \ce[-3] 
$}
\put(18,-23){\line(1,0){106}}
\put(18,-23){\line(0,1){16}}                               %one out arrow
\put(124,-23){\vector(0,1){16}}
\put(23,-19){\scriptsize$\sf \Pop_4$}
\end{picture}
$$
and in this case
\begin{equation}\label{P4fact}
\Pop_4:= {\sf aE}^*\circ {\sf aE}.
\end{equation}
We see here that in dimension 4 the operator $\Pop_4$ factors through
$\operatorname{\sf aE}$. This aspect generalises. 
\begin{lemma}\label{factor}
On any conformally flat manifold of even dimension $n$, the operator
$\Pop_{n+2}:\ce[1]\to \ce[-n-1]$ can be written as a composition
$$
\Pop_{n+2}= \operatorname{\sf H}\circ \operatorname{\sf aE},
$$ where $\operatorname{\sf H}:\ce^{1,1}[1]\to \ce[-n-1]$ is a
differential operator.
\end{lemma}
The operator $\operatorname{\sf H}:
\ce^{1,1}[1]\to \ce[-n-1]$ is not conformally invariant (except in dimension~2),   but must be conformal on the range of
$\operatorname{\sf aE}$.
 In fact, analogous factorisations as here hold for all
the long operators (and in fact stronger results are available); a
general approach to establishing this is discussed in~\cite{GP-obstrn}. Here, there is a simple proof:
\begin{proof}[Proof of Lemma \ref{factor}] Any conformally flat manifold is locally conformally isomorphic to the sphere and such local maps give an  injection between conformally invariant operators  and differential operators on the sphere that intertwine the conformal group action. 

In dimension 2 it is straightforward to verify~\nn{P4fact}. For higher
 dimensions, from the construction of the conformal Laplacians on
 conformally flat manifolds given in~\cite[Section 4.2]{GoSrni99} (see also
 \cite{GoPetCMP}) one has
$$ 
X_A \Pop_{n+2} =-\Pop_n\circ D_A ,
$$
where $\Pop_n$ is the order $n$ conformal Laplacian. On the other hand from the same sources, or as also discussed below in Section \ref{EGJMS}, $\Pop_n$, as an operator on tractor fields (of weight~0), factors through the tractor connection:
$$
\Pop_n={\mathcal G}\circ \nabla,
$$
for some differential operator ${\mathcal G}$.
So 
$$
X_A \Pop_{n+2} =-\, {\mathcal G}\circ\nabla\circ D_A .
$$
But a straightforward calculation verifies that 
$\nabla\circ D_A$ factors through the operator $\operatorname{\sf aE}$ (with this on the right). 
 In fact this underlies the {\em construction} of the tractor connection in \cite{BEG}.  
\end{proof}

\newcommand{\cK}{\mathcal{K}}
Let us write $\cK^{1,1}[1]$ for the space of smooth sections in the kernel of 
$$
\operatorname{{\sf Cod}}:\ce^{1,1}[1]\to \ce^{2,1}[1]
$$ in even dimensions $n\geq 4$, and for the kernel of ${\sf
  aE}^*\circ \tstar$ in dimension 2. Then we have the following:

\begin{proposition}\label{flatcase}
On any even dimensional conformally flat Riemannian $n$-manifold there
is a differential operator ${\sf H}:\ce^{1,1}[1]\to \ce[-n-1]$ that
upon restriction to $\cK^{1,1}[1]\subset \ce^{1,1}[1]$ is
conformally invariant.
\end{proposition}
\begin{proof}
Since the underlying BGG complex is locally exact, in each case, it
follows that locally $\cK^{1,1}[1]$ is the image of
$\operatorname{aE}$. Thus the result follows  from the Lemma.
\end{proof}

Now recall that on any conformal hypersurface $\IIo$ is a section of
$\ce^{1,1}[1]$. For a hypersurface in Euclidean $(n+1)$-space this is
in the kernel of the operator $\operatorname{\sf Cod}$ (see Equation~\nn{CODIIO}).  On the other
hand on a conformally flat $n$-manifold with $n$ even, given any
section $\mathring{K}$ of $\cK^{1,1}[1]$ we may form
$$
{\sf B}_{\mathring{K}}:= \operatorname{\sf H}(\mathring{K}). 
$$ 
From Proposition~\ref{flatcase} it follows that this is a
conformally invariant section of $\ce[-n-1]$. This is in
general nontrivial; the last claim following from the fact that
$\Pop_{n+2}$ is elliptic (and so has finite dimensional kernel). Thus we see that in this sense the existence of
the obstruction density in even dimensions, and the absence of its
linearisation in odd dimensions ({\it i.e.}, as in Proposition
\ref{Bnature}) is nicely compatible with linear conformal differential
operator theory. Note that for comparison with Proposition
\ref{Bnature}, at leading order $\operatorname{\sf aE} (\dot{\si})$ is the
linearisation of $\IIo$. So, for
embedding variations of the standard sphere in Euclidean space, the  linearisation of the obstruction density~$\dot {B}_\sigma\big|_{S^n}$ is a non-zero multiple of ${\sf B}_{{\sf aE}(\dot\sigma)}$, in concordance with~\nn{B=Pndotsigma}.

Let us say that a conformally invariant curvature invariant of a
hypersurface is a {\em (hypersurface) fundamental curvature quantity}
if it has a non-trivial linearisation, with respect to variation of
the hypersurface embedding, when evaluated on the conformal class of the round sphere
embedded in Euclidean space.   
We have established
above that the obstruction density is a fundamental curvature quantity
in this sense. The trace free second fundamental form is also; its
linearisation being the BGG operator $\operatorname{\sf aE}: \ce[1]\to
\ce^{1,1}[1]$. In fact these are effectively the only such invariants.
In the following statement we ignore the possibility of multiplying an
invariant by a non-zero constant.

\medskip

We are ready to prove Theorem \ref{uniqui} mentioned in the introduction.
\begin{proof}[Proof of Theorem \ref{uniqui}] \label{ui}
Defining hypersurface invariants as we do, it is easily verified that
any linearisation of a hypersurface invariant is a conformally invariant
differential operator between irreducible bundles, with domain bundle
$\ce[1]$. Thus the result is immediate from the known classification
of such operators mentioned above.  
\end{proof}
\noindent For hypersurfaces of dimension~2 there is $\cB$, $\IIo$,
and also the invariant $\star\IIo$, and these give the
full set fundamental curvature invariants. 

\begin{remark}\label{properties}
It follows from the classification of conformally invariant operators on the sphere and the above discussion
that, except for hypersurface dimension~2, the obstruction density
cannot be written, even at leading order, as a conformally invariant operator acting on the trace-free second fundamental form.
\end{remark}

Finally we point out that the discussion here gives a precise sense in
which the obstruction density is a scalar analogue of the situation
with the Fefferman-Graham obstruction tensor: The latter exists in
even dimensions, and its linearisation can be understood, via the
appropriate BGG diagrams, in a manner exactly parallel to the
treatment here for the obstruction density, see \cite[Section
  2]{GP-obstrn}. There is also an analogue of Theorem \ref{uniqui},
see \cite[Theorem 1.2]{GrHir}.

\section{Naturality of the obstruction density and proliferating invariants}\label{canda}

By construction the obstruction density depends only on the
data of the conformal embedding $\Sigma\hookrightarrow M$, however it remains
to prove that the obstruction density is a hypersurface conformal
invariant in the sense of Definition \ref{chi-def}. We also establish
here how the results above may be used to construct other hypersurface
conformal invariants.

\subsection{Hypersurface conformal invariants}\label{invts}

To construct invariants holographically we need, as a tool, a broader
class of invariants that we term coupled invariants.  For Riemannian
manifolds $(M,g)$, {\em scalar Riemannian invariants} (as
in~\cite{AtiyahBottPatodi}) may be thought of as pre-invariants, as
defined by (i), (ii), (iii) of the Definition \ref{R-invtdef}, if the
dependence on $s$ is required to be trivial. If $s$ is a scalar
function then we will also talk of {\em coupled invariants}. This
means the same as a pre-invariant (again as in the Definition
\ref{R-invtdef}) where we do {\em not allow} the inclusion of $||
\boldsymbol{d} s ||_g^{-1}$ in (iii) of Definition \ref{R-invtdef}
(but do allow polynomial dependence on the jets of $s$).  These
notions adapt easily to tensor-valued coupled invariants.

Then Riemannian invariants or coupled invariants are conformal
invariants or coupled conformal invariants if we have the analogue of
Definition \ref{chi-def}. That is:
\begin{definition}\label{ci-def}
A {\em weight~$w$ coupled conformal covariant} is a coupled
Riemannian invariant~$P(s,g)$ with the property that
$P(\Omega^u s,\Omega^2 g)= \Omega^w P(s,g)$, for any smooth positive
function~$\Omega$ and $u,w\in \mathbb{R}$.
Any such covariant determines an invariant
section of~$\ce[w]$ that we
shall denote~$P(\si, \bg)$, where~$\bg$ is the conformal
metric of the conformal manifold~$(M,[g])$ and $\si\in \Gamma(\ce[u])$. We shall say that~$P(\si,\bg)$ is a {\em
  coupled conformal invariant} of weight~$w$, or simply a {\em
   conformal invariant} of weight~$w$ if the dependence on $\si$ is trivial. 
\end{definition}

%%edz[Make next para a Theorem: -- basically result is natural as
%%  immediate from Theorem \ref{main-calc} above.]  The unit conformal
%% scale on an ambient manifold enables a ``holographic'' study of
%% extrinsic as well as intrinsic hypersurface conformal geometry: The
%% key ingredient is Theorem~\ref{obstr}, which can be used to
%% proliferate natural invariants of the conformal hypersurface
%% structure~$(M,\cc,\Sigma)$ (see Definition~\ref{chi-def}).  

Now the key idea is to consider such coupled invariants when $\si$ is a conformal unit defining density $\bar\si$ for a
conformally embedded hypersurface $\Sigma \hookrightarrow M$.  According to
Theorem~\ref{obstr}, the conformal unit scale~$\bar\si$ is determined
by the data $(M,\cc,\Sigma)$, uniquely modulo~$\O (\si^{d+1})$. Thus
if, at each point, a coupled invariant $P(\bar\si,\bg)$ of $\cc$ and~$\bar\si$ depends on $\bar\si$ only through its $d$-jet, then
$P(\bar\si,\bg)|_\Sigma$ is conformally invariant in that it depends
only on the data ~$(M,\cc,\Sigma)$. In fact $P(\bar\si,\bg)|_\Sigma$
is a hypersurface conformal invariant in the sense of
Definition~\ref{chi-def}. Given its conformal invariance, to show this
we only need to show that there is an expression for $P$, with the  form
described in Definition \ref{R-invtdef}. In practice this is achieved by
showing that the covariant derivatives of~$\bar\si$ may be replaced
with expressions involving the derivatives of the second fundamental
form, the conormal, and the ambient Riemannian curvature,
{\it cf.} Expression \ref{parcon}.
  
A technical definition is needed for the 
 main lemma: Given a hypersurface $\Sigma$, and any
defining function $s$ for $\Sigma$, we say a linear differential
operator ${\mathcal D}$ has {\it transverse order} at most $\ell\in
\mathbb{Z}_{\geq 0}$, along $\Sigma$, if ${\mathcal D}\circ
s^{\ell+1}$ acts as zero along $\Sigma$ (where $s^{\ell+1}$ is viewed
as a multiplication operator).

\begin{lemma}\label{main-calc}
Suppose that~$\bar{\si}$ is a conformal unit defining density for a
hypersurface~$\Sigma$ in a conformal manifold~$(M^d,\cc)$, with~$d\geq 3$.
 If~$g \in \cc$ and ~$k\leq d$ is a positive integer, then the quantity
$$
\nabla^k \bar{\si}|_\Sigma ,\quad \mbox{where} \quad \nabla=\mbox{\rm Levi-Civita of } g,
$$ may be expressed as~$\nablab^{k-2}\II$ plus terms involving partial
contractions of the Riemannian curvature, its covariant derivatives,
and covariant derivatives of $\bar\si$ with transverse order at most $k-1$.
%% a linear combination of
%% terms where each term is a homogeneous polynomial in various
%% derivatives~$0\leq m \leq k-2$, times a partial contraction involving
%% the conormal~$n$, $\nablab^{\ell}\II$ for~$ 0\leq \ell\leq k-3$, and
%% the Riemannian curvature~$R$ and its covariant derivatives (to order
%% at most~$k-3$).
\end{lemma} 

\noindent The proof of this is straightforward and so we
omit it here. However the main ingredient to one effective approach is
Lemma \ref{rho-engine} below (which in fact yields much more).

%In summary we obtain the following result.
We can now state the main result for constructing conformal hypersurface invariants from a conformal unit defining density:

%% Up to the order that~$\bar{\si}$ is uniquely
%% determined, the coupled conformal invariants of the conformal
%% structure and the scale~$\bar \si$ are automatically natural
%% invariants of~$(M,\cc,\Sigma)$. Such invariants are easily constructed
%% using the ambient conformal tractor calculus applied to~$(M,\cc)$
%% and~$\bar\si$. Formul\ae\ for conformal hypersurface invariants
%% obtained by restricting invariants of the ambient ASC structure
%% to~$\Sigma$ are termed {\it holographic formul\ae}.

\begin{theorem}\label{invtthm}
Let $\bar \sigma$ be a conformal unit defining density for a hypersurface $\Sigma\hookrightarrow M$.
Suppose that $P$ is a weight $w$ coupled conformal invariant of $(M,\cc,\bar{\si})$, as in Definition~\ref{ci-def}, such that at each
point it depends on at most the $d$-jet of $\bar{\si}$. Then the restriction of $P$ to $\Si$ is a conformal hypersurface invariant of weight $w$.
\end{theorem}
\begin{proof} That $P$ depends only on the conformal embedding is immediate 
from Theorem~\ref{obstr}. Then  naturality and the other properties 
follow from Lemma
\ref{main-calc} above combined with an obvious induction. 
\end{proof}

\subsection{Naturality of the obstruction density}\label{natB} Theorem 
\ref{invtthm} does not immediately imply that $\cB$ is a conformal
hypersurface invariant, because $\cB$ depends on $\bar\si$ to order
$d+1$. For this we need further detail from the equation defining
$\bar\si$ and $B$, namely
\begin{equation}\label{nid}
n^2= 1-2\rho \sib + \sib^d B  \quad \Leftrightarrow \quad I_{\bar\sigma}^2=1+\sib^d B,
\end{equation}
for some smooth $B$, where
~$n:=n_{\bar{\si}}$ is used to denote~$\nabla \bar{\si}$, and ({\it cf.}~\nn{sctrac-def})
 \begin{equation}\label{Iform}
[I^A_{\sib}]:=[\hD^A \sib] = \left( \begin{array}{c} \sib\\
n_a\\
\rho \end{array}\right)\, ,\qquad \rho:=\rho(\bar\si)=-\frac{1}{d}(\Delta \sib + \J \sib)\, .
\end{equation}
Such a defining density exists by Theorem \ref{obstr} and is canonical
to~$\O(\sib^{d+1})$. The obstruction density is
$\cB=B|_\Sigma$. Differentiating equation \nn{nid} suitably and using the definitions and relations in \nn{Iform} we obtain the following technical result (see~\cite[Section 3.1]{GW-willII} for further detail):
%\edz{RS: I guess the
%    obstruction density is the missing trace in the usual
%    Poincare-Einstein calculation}
\begin{lemma}\label{rho-engine} For integers~$2\leq k \leq d$
\begin{equation}\label{rho-step}
\begin{split}
\frac{1}{2} \nabla_n^k I^2_{\bar\sigma} + (d-k) \nabla_n^{k-1} \rho  
\stackrel{\Sigma}{=}
-\nabla_n^{k-1}\big( \gamma^{ab}\nabla_an_b \big) & - 
(k-1) \big[\nabla_n^{k-2}\big( \J+2\rho^2 \big)+(k-2)\rho\nabla_n^{k-2}\rho\big]\\[1mm]& +  \makebox{\rm LTOTs}\, ,
\end{split}
\end{equation}
where~$\makebox{\rm LTOTs}$ indicates additional terms involving lower transverse-order derivatives of~$\sib$. 
In particular, for  ~$2\leq k \leq d-1$ we have
\begin{equation}\label{builder}
\begin{split}
\nabla_n^{k-1} \rho  
\stackrel{\Sigma}{=}
-\frac{1}{d-k}\Big(\nabla_n^{k-1}\big( \gamma^{ab}\nabla_a n_b \big) &+ 
(k-1) \big[\nabla_n^{k-2}\big( \J+2\rho^2 \big) +(k-2)\rho\nabla_n^{k-2}\rho\big] \Big)\\&+ \makebox{\rm LTOTs}\, ,\end{split}
\end{equation}
while 
\begin{equation}\label{Bform}
\cB \stackrel{\Sigma}{=} -\frac{2}{d!}\Big(
\nabla_n^{d-1}\big( \gamma^{ab}\nabla_a n_b \big) + 
(d-1)\big[ \nabla_n^{d-2}\big( \J+2\rho^2 \big)
+(d-2)\rho\nabla_n^{d-2}\rho\big] 
 \Big) + \makebox{\rm LTOTs}\, .
\end{equation}
\end{lemma}
 Using this, Lemma \ref{main-calc},  and a
straightforward induction,  then shows that
the obstruction density $\cB$ may be expressed as a linear combination
of partial contractions involving $\nablab^{\ell}\II$ for~$ 0\leq
\ell\leq d-1$, and the Riemannian curvature~$R$ and its covariant
derivatives (to order at most~$d-3$) and the undifferentiated
conormal~$n$. Given its conformal invariance by construction we thus have:
\begin{theorem}\label{Bnatthm}
In each dimension $d\geq 3$ the obstruction density $\cB$ is a
conformal hypersurface invariant.
\end{theorem}

\section{Extrinsically coupled conformal Laplacians and a holographic formula for $\cB$}
%\section{Invariant operators and holographic formul\ae}
\label{invops}

In this section our main aim is to construct conformally invariant
powers of the Laplacian on $\Sigma$ that are canonically determined by
the structure $(M,\cc,\Sigma)$. In the cases where corresponding
intrinsic GJMS operators exist, these differ by their dependence on
the extrinsic geometry of the conformal embedding
$\Sigma\hookrightarrow M$. We give holographic formulae for these.
These are then applied to construct a holographic formula for the
obstruction density.

\subsection{Extrinsic conformal Laplacians}\label{EGJMS}

We shall construct distinguished conformally invariant hypersurface operators canonically determined by $(\cc,\Sigma)$.
Our starting point is the Laplacian-type operators constructed in~\cite{GW} where  it is shown that 
$$
{\mathcal P}_k^\sigma:\Gamma\Big(\ct^\Phi M\Big[\frac{k-d+1}{2}\Big]\Big)\rightarrow \Gamma\Big(\ct^\Phi M\Big[\frac{-k-d+1}{2}\Big]\Big)\, ,\quad k\in {\mathbb Z}_{\geq1}
$$
defined by
\begin{equation}\label{holP}
{\mathcal P}^\sigma_k:=\Big(\!-\frac{1}{I_\sigma^2}\, I_\sigma .D\Big)^k,
\end{equation}
is tangential for {\it any} defining density~$\sigma$.  However, according to Theorem~\ref{obstr},~$\bar \si$ is uniquely determined by
$(M,\cc,\Sigma)$, modulo terms of order~$\bar\si^{d+1}$. 
Hence, specializing the defining density to be unit conformal, we obtain extrinsic conformal Laplace operators:

\begin{theorem}\label{ecL}
The operator
$${\mathcal P}^{\bar\sigma}_k:\Gamma\big(\ct^\Phi M\big[\frac{k-d+1}{2}\big]\big)\rightarrow \Gamma\big(\ct^\Phi M\big[\frac{-k-d+1}{2}\big]\big)\, ,$$
is a tangential differential operator. Moreover, for $k\leq d-1$, 
along~$\Sigma$ this is determined canonically by the 
data~$(M,\cc,\Sigma)$, and 
when~$k$ is even has leading term $$(-1)^k \big((k-1)!!\, \big)^2\, \big(\Delta^{\!\!\top}\big)^{\frac k2}\, .$$
Thus ${\mathcal P}_k^{\bar\sigma}$ determines a differential  operator
$$\Pop_k:\Gamma\big(\ct^\Phi M\big[\frac{k-d+1}{2}\big]\big)\Big|_\Sigma\rightarrow \Gamma\big(\ct^\Phi M\big[\frac{-k-d+1}{2}\big]\big)\Big|_\Sigma\, ,$$
which we shall call an
extrinsic conformal Laplace operator. 
\end{theorem}

\begin{proof}
As stated above, ${\mathcal P}^\sigma_k$ is tangential
for any~$\sigma$ and hence for $\bar \sigma$. 
Moreover, 
by Theorem~\ref{obstr},~$\bar \si$ is uniquely determined by
$(M,\cc,\Sigma)$, modulo terms of order~$\bar\si^{d+1}$.  
So we check how the operator $\bar y=y_{\bar\sigma}$ (see Definition~\ref{xhy}) changes when replacing $\bar \sigma$ with $\bar\si + \bar\si^{d+1} A$, for some smooth, weight $-d$, density $A$.
From the
formula~\nn{hDsip}, we have 
$$
\hD_B (\bar\si + \bar\si^{d+1} A) 
= I_B+ (d+1)\bar\si^{d} I_B  A + \bar\si^{d-1}X_B C + \cO (\bar\si^{d+1})\, ,
$$ 
for some
density~$C$. Thus, using also expression~\nn{ddens} for~$I^2$, we see that
each operator~$\bar y$ in the composition~$\bar y^{k}$ is uniquely determined by
$(M,\cc,\Sigma)$ up to the addition of~$\bar\si^{d-1} E$,  for
some linear operator~$E$.  Hence using the first identity of Equation~\nn{yk}, because $k\leq d-1$, it follows  that  the operator~$\bar y^{k}$ is unique modulo the addition of a linear operator
which vanishes along~$\Sigma$. Thus $ \bar y^{k}\big|_\Sigma$~is
uniquely determined by~$(M,\cc,\Sigma)$ as claimed.

Finally, it follows from~\cite[Proposition 4.4]{GW} that when~$k$
is even,~$\bar y^{k}$ has leading term~$(-1)^k \big((k-1)!!\big)^2(\Delta^{\!\top})^{k/2}$ (as an operator
on ambient  tractor fields along~$\Sigma$).
\end{proof}

\begin{remark}
When~$(M,\cc,\sigma)$ is an AE structure 
and~$k$ is even, the operator ${\mathcal P}_k^\sigma$ gives a holographic formula for conformally invariant GJMS-type operators~\cite{GW}.
These take the form
$
\scalebox{.92}{$(-1)^k \big((k-1)!!\big)^2$}\   \big(\Delta^{\!\!\top}\big)^{\frac k2} \ +\  LOT
$,
where ``$LOT\,$'' denotes lower order derivative terms. For $k\leq d-3$, these are built from the intrinsic hypersurface Levi-Civita connection and its curvature  (a slightly stronger statement is available for 
densities~\cite{GJMS}). If $d$ is even and the~AE structure is even (in the sense of~\cite{FGrnew}) then this holds\footnote{In~\cite{GW}, evenness of the AE structure for odd~$n$ was assumed but not mentioned.} for all even~$k$. 
Relaxing the AE 
condition as in Theorem~\ref{ecL},
the terms~$LOT$ include extrinsic hypersurface invariants. Moreover  for~$k$ odd the operators~$\Pop_k$
are no longer trivial along~$\Sigma$,  in contrast to the AE case.
\end{remark}

Because they are tangential, the extrinsic conformal Laplacians have natural formul\ae\  involving tangential derivatives~$\nabla^\top\!$, as recorded in the following:

\begin{proposition}\label{Pnt}
Let $k\leq \dim\Sigma$. Then, in a given choice of scale, 
the  extrinsic conformal Laplacian $\Pop_k$ has a formula 
$$
\Pop_k={\mathcal A}_0+\sum_{j=1}^k{\mathcal A}^{a_1\ldots a_j}
\nabla_{a_1}^\top\cdots \nabla_{a_j}^\top \, ,
$$
where the scalar ${\mathcal A}_0$ and the tensors ${\mathcal A^{a_1\ldots a_k}}$ are 
natural formul\ae\  given by polynomial expressions in first and second fundamental forms, and boundary Levi-Civita derivatives thereof, as well as ambient curvatures and their ambient  Levi-Civita derivatives.
Moreover, when $k=\dim \Sigma$ the scalar term ${\mathcal A}_0=0$ is absent and 
$$
\Pop_{\dim \Sigma} ={\mathcal G}\circ \nabla^\top\, ,
$$
for some tangential operator ${\mathcal G}$.
\end{proposition}

\begin{proof}
Using the fact that the operators are tangential 
in each case, it follows easily that there is a formula involving only tangential derivatives $\nabla^\top$. It is straightforward to check that this can be achieved using the calculus developed
above; the claim concerning the natural formula
then follows. For the special case $k=\dim \Sigma$, the Thomas D-operator on the right in the defining formula~\nn{holP} factors through an ambient Levi-Civita connection on the right. 
Thus, when the preceding argument is applied to this case, it follows that~${\mathcal A}_0=0$.
\end{proof}

\begin{remark}
In the above, and in the proof of Theorem~\ref{alllaps} below, we could equivalently trade $\nabla^\top$ for the intrinsic tractor connection coupled to the  ambient tractor connection, pulled back  to~$\Sigma$. 
\end{remark}

Finally, in this section we show that there are extrinsic conformal Laplacian operators of all (even) order.

\begin{theorem}\label{alllaps}
Let $\dim(\Sigma)$ be even and $k\in  2{\mathbb N}$.
Then there exists a canonical differential  operator on $\Sigma$
$$\Pop_k:\Gamma\big(\ct^\Phi M\big[\frac{k-d+1}{2}\big]\big)\Big|_\Sigma\rightarrow \Gamma\big(\ct^\Phi M\big[\frac{-k-d+1}{2}\big]\big)\Big|_\Sigma\, ,$$
with leading term $(\Delta^\top)^{\frac k2}$,
determined  by the data $(M,\cc,\Sigma)$.
\end{theorem}

\begin{proof}
When $k\leq \dim(\Sigma)$, the theorem is simply Theorem~\ref{ecL} specialised to even dimensional hypersurfaces. For all higher orders $k> \dim(\Sigma)$, we consider the operator
\begin{equation}\label{Pallk}
\Pop_k:=D^{{}^{\rm T}}_{A_1}\cdots D^{{}^{\rm T}}_{A_{\ell}}\,  \Pop_{\dim(\Sigma)}\, 
 D^{{}^{\rm T}}{\!}^{A_\ell}\cdots D^{{}^{\rm T}}{\!}^{A_1}\, ,
\end{equation}
where $\ell:=\frac{k-d+1}{2}$. Because this is built from {\it
  tangential} Thomas D-operators and the tangential operator
$\Pop_{\dim(\Sigma)}$, it depends only on the data $(M,\cc,\Sigma)$.
Moreover, from Theorem~\ref{ecL}, and the definition of $D^{{}^{\rm
    T}}$, it is clear that the leading derivative term of the
  operator~$\Pop_{\dim(\Sigma)}\, D^{{}^{\rm T}}{\!}^{A_\ell}\cdots
  D^{{}^{\rm T}}{\!}^{A_1}$ in the above display is a non-zero
  constant multiple of $ X^{A_\ell}\cdots X^{A_1} \,
  \big(\Delta^{\!\top}\big)^{\frac k2}\, , $ so
 \begin{equation}\label{LOTs}\Pop_{\dim(\Sigma)}\, 
 D^{{}^{\rm T}}{\!}^{A_\ell}\cdots D^{{}^{\rm T}}{\!}^{A_1}\propto X^{A_\ell}\cdots X^{A_1} \, \big(\Delta^{\!\top}\big)^{\frac k2}+{\rm LOTs}\, ,
 \end{equation}
where LOTs stands for some lower derivative operator.
 
In addition an easy calculation and induction establishes the 
following operator identity, valid acting on weight~$-\dim(\Sigma)-\ell$ tractors:
% follows immediately from   Corollary~\ref{bDXT}:
$$
D^{{}^{\rm T}}_{A_1}\cdots D^{{}^{\rm T}}_{A_{\ell}} X^{A_\ell}\cdots X^{A_1}\stackrel\Sigma=\Big[\prod_{i=1}^{\ell}i(d+2i-3)\Big] \, {\rm Id}\neq 0\, . 
$$
Hence it follows that the leading derivative term in $\Pop_{\dim(\Sigma)}\, 
 D^{{}^{\rm T}}{\!}^{A_\ell}\cdots D^{{}^{\rm T}}{\!}^{A_1}$ produces a non-zero contribution to $\Pop_k$ proportional to $(\Delta^{\!\top})^{\frac k2}$. Moreover, a weight argument shows that this is the highest possible order of derivative contribution to $\Pop_k$. It only remains, therefore, to show that lower order derivative contributions LOTs involve curvatures and therefore 
  cannot conspire in the full formula for $\Pop_k$ in  Equation~\nn{Pallk} to produce further leading order terms. For this we recall
 that the classification of conformal operators on the sphere $S^{d-1}$ ({\it cf}.~\cite{EastwoodRice,Esrni,GoSrni99}) yields the operator identity on intrinsic, weight $-\ell$ tractors 
 $$
 \bar \Pop_{d-1} \bar D^{A_\ell}\cdots \bar D^{A_1}=\alpha \, X^{A_\ell}\cdots X^{A_1}\, \bar \Pop_{d+2\ell-1}\, ,
 $$
where $\bar \Pop_k$ denotes the usual (tractor twisted) conformal Laplacian, $\bar D$ is the intrinsic Thomas D-operator and $\alpha$ is a non-zero constant. Therefore, the lower order terms LOTs in~\nn{LOTs} all involve curvatures at least linearly.
 \end{proof}

\begin{remark}
The above proof proceeds {\it mutatis mutandis}
if one wishes to  replace the tangential Thomas D-operators in Equation~\nn{Pallk} by the intrinsic Thomas D-operator twisted by the ambient tractor connection.
\end{remark}

\subsection{ASC obstruction density}\label{ASCObst}

We now derive a holographic formula giving the main structure
of the obstruction density; in particular this shows the {\it r\^ole} of the extrinsic conformal Laplacians derived above  and 
 facilitates its computation (see~\cite{GGHW} where the holographic formula is applied to volumes embedded in four-manifolds).

\begin{theorem}\label{holoB}
Let~$\bar \sigma$ be a unit conformal defining density. Then, the ASC obstruction density~$\B$
is given by the holographic formula
\begin{equation}\label{Bhol}
  \B= \frac{2}{d!(d-1)!}\, \bar D_A \Big[\Sigma^A_B\Big(\Pop_{d-1} N^B + (-1)^{d-2} \big[\bar I\cdot D^{d-2}(X^B K_{\rm ext})\big]\Big)\Big|_\Sigma\Big]\, ,
\end{equation}
where~$K_{\rm ext}:=P_{AB}P^{AB}$ and~$P^{AB}:=\hD^A\bar I^B$, and $N^B$ is any extension of the normal  tractor off~$\Sigma$. 
\end{theorem}
\begin{proof}
First we employ the identity 
$$
I_\sigma.D I_\sigma^A = \frac{1}{2} D^A I_\sigma^2 + X^A K_{\rm ext}\, ,
$$
which is easily verified and valid for any defining density~$\sigma$.
 Next we deduce  from Lemma~\ref{actonstuff}, that
$$
D^A (\bar\si^d B)=-d (d-1) \bar \si^{d-2} X^A  B +\O(\bar \si^{d-1})\, ,
$$
where the the unit conformal defining density property has been used to replace~$\bar I^2$ by unity on the right hand side.
Hence, along~$\Sigma$ we have
\begin{equation}\label{below}
\big(\bar I\cdot D^{d-2}\circ \bar I\cdot D\big)\,  \bar I^A = -\frac{d(d-1)}{2} \big[\bar I\cdot D^{d-2},\bar \si^{d-2}\big] X^A B + \bar I\cdot D^{d-2} [X^A K_{\rm ext}]\, .
\end{equation}
Here we used $\bar I\cdot D^{d-2}\circ \bar\sigma^{d-1}\stackrel\Sigma=0$ by virtue of the algebra~\nn{sl2} and the conformal unit defining density property~\nn{I2def}.

Along~$\Sigma$, the above commutator in~\nn{below} can be replaced by $(-1)^{d-1}[x^{d-2},y^{d-2}]$ where
the operators~$x= \si$ (viewed as a multiplicative operator) and~$y=-\frac{1}{I^2} I\cdot D$, introduced in Section~\ref{sl2structure}, obey the standard~${\frak sl}(2)$ algebra
$[x,y]=h$,
for any defining density~$\sigma$.
A simple inductive argument shows that the relation
$$
[x^k,y^k]= (-1)^{k+1} k!\, h(h+1)(h+2)\cdots (h+k-1)+x \, F
$$
holds in the~${\frak sl}(2)$ enveloping algebra for some polynomial of~$F$ in the generators~$\{x,h,y\}$.
In the present case, the operator~$h$ acts on weight~$w$ tractors by multiplying by~$d+2w$.
Since~$X^AB$ has weight~$1-d$ we have $h X^A B=(2-d) X^AB$. Hence, along~$\Sigma$, we have
$$
[\bar I\cdot D^{d-2},\bar \si^{d-2}]X^A B=(d-2)! (2-d)(3-d)\cdots (-1) X^A B=(-1)^{d-2}[(d-2)!]^2X^AB\, .
$$
Finally, we apply the identity~\nn{DXT}
(specialized to  tractors along $\Sigma$)
to the quantity $\bar D_A \big[ (X^A B )|_\Sigma\big]= \bar D_A (\bar X^A B_\Sigma)$. Elementary bookkeeping then gives the quoted result.
\end{proof}

\begin{remark}
It is easily seen that the first term on
the right hand side of the holographic formula~\nn{Bhol} vanishes when
the second fundamental form is zero.  Hence, it is interesting to ask
whether the obstruction density vanishes for totally umbilic
hypersurfaces.  Indeed for embedded surfaces and volumes the second
term on the right hand side of Equation~\nn{Bhol} also vanishes.
Hence, the obstruction density is zero for totally umbilic embeddings
hypersurfaces of dimensions 2 and 3. Whether this vanishing extends to
higher dimensional hypersurfaces is an open problem.
\end{remark}

%% Proposition~\ref{GJMS23} gives a compact formula for the extrinsic
%% Laplacian appearing above when~$d=4$, thus the remaining difficulty in
%% computing the~$d=4$ obstruction density in curved ambient spaces is
%% calculating two normal derivatives of the canonical extension~$K_{\rm
%%   ext}$ of the rigidity density~$\IIo_{ab}\IIo^{ab}$. Since one normal
%% derivative of the canonically extended trace-free second fundamental
%% form is closely related to the Fialkow tensor, this boils down to
%% computing one normal derivative of the corresponding extension of the
%% Fialkow tensor. This is the next natural example in the general
%% program of proliferating natural invariants of the conformal
%% hypersurface structure~$(M,\cc,\Sigma)$ discussed in
%% Section~\ref{invts}. We will report on that computation and
%% obstruction densities in higher dimensions elsewhere~\cite{GGHW}.

\appendix

\section{Proof of Proposition~\ref{leib-fail}} \label{FG}

First a technical Lemma.  
\begin{lemma}\label{leib0}
Let~$T_i\in \Gamma(\cT^\Phi M)[w_{i}]$ for~$i=1,2$ and~$h_{i}:=d+2w_{i}$,~$h_{12}:=d+2w_1+2w_2-2$. Then
$$
-2 X^A (D_B T_1)\,  (D^B T_2)= h_1 h_2\,  D^A(T_1T_2) -h_{12} \, \big(h_2\, (D^A T_1) T_2 + h_1 T_1 (D^A T_2)\big)\, .
$$
\end{lemma}
\begin{proof}
To prove Lemma~\ref{leib0} we employ the Fefferman--Graham ambient
metric construction of the standard tractor
bundle~\cite{CapGoamb,GoPetCMP}. Our notations are those of~\cite[Section~6]{GW}. In particular, the Thomas D-operator is a restriction
of the following operator on sections of the ambient tensor bundle:
$$
\aD_A= \aNd_A(d+2\aNd_{\sX}-2) +\aX_A\, \bs{\Delta}.
$$
Acting on a product of ambient tensors~$\tilde T_1\tilde T_2$ of homogeneities~$w_1$ and~$w_2$ (so that~$(\aNd_{\sX}-w_1)\, \tilde  T_1=0=
(\aNd_{\sX}-w_2)\, \tilde  T_2$) we have
\begin{eqnarray*}
\aD_A \big(\tilde T_1 \tilde T_2\big)&=&(d+2w_1+2w_2-2)\big((\aNd_A \tilde T_1)\, \tilde T_2+\tilde T_1\, (\aNd_A \tilde T_2)\big)\\[.5mm]
&-&\aX_A\big((\bs{\Delta} \tilde T_1)\, \tilde T_2 + 2 \, (\aNd_B \tilde T_1)(\aNd^B \tilde T_2)+\tilde T_1\, (\bs{\Delta} \tilde T_2)\big)\, ,
\end{eqnarray*}
so that
\begin{equation*}
\begin{aligned}
(d+&2w_1-2)(d+2w_2-2)\aD_A \big(\tilde T_1 \tilde T_2\big)\\[3mm]
&=(d+2w_1+2w_2-2)\big((d+2w_2-2)\, (\aD_A \tilde T_1)\, \tilde T_2+(d+2w_1-2)\, \tilde T_1\, (\aD_A \tilde T_2)\big)\\[.5mm]
&+(d+2w_1+2w_2-2)\, \aX_A\big((d+2w_2-2)\, (\bs{\Delta} \tilde T_1)\, \tilde T_2 + (d+2w_1-2)\,  \tilde T_1\, (\bs{\Delta} \tilde T_2)\big)\\[.5mm]
&-(d+2w_1-2)(d+2w_2-2)\, \aX_A\big((\bs{\Delta} \tilde T_1)\, \tilde T_2 + 2 \, (\aNd_B \tilde T_1)(\aNd^B \tilde T_2)+\tilde T_1\, (\bs{\Delta} \tilde T_2)\big)\\[3mm]
&=(d+2w_1+2w_2-2)\big((d+2w_2-2)\, (\aD_A \tilde T_1)\, \tilde T_2+(d+2w_1-2)\, \tilde T_1\, (\aD_A \tilde T_2)\big)\\[.5mm]
&-2\aX_A\, (\aD_B \tilde T_1)(\aD^B \tilde T_2) + {\cO}(\aX^2)\, .
\end{aligned}
\end{equation*}
\end{proof}

\newcommand{\msn}[2]{\href{http://www.ams.org/mathscinet-getitem?mr=#1}{#2}}
%number
\newcommand{\hepth}[1]{\href{http://arxiv.org/abs/hep-th/#1}{arXiv:hep-th/#1}}
%number
\newcommand{\maths}[1]{\href{http://arxiv.org/abs/math/#1}{arXiv:math/#1}}
%number
\newcommand{\mathph}[1]{\href{http://lanl.arxiv.org/abs/math-ph/#1}{arXiv:math-ph/#1}}
\newcommand{\arxiv}[1]{\href{http://lanl.arxiv.org/abs/#1}{arXiv:#1}}

\end{document}